\documentclass[12pt]{article}

\usepackage{tikz}
\usepackage{amsmath,amsthm,amssymb,xcolor,multicol,url}

\theoremstyle{plain}
\newtheorem{theorem}{Theorem}
\newtheorem{lemma}[theorem]{Lemma}
\newtheorem{corollary}[theorem]{Corollary}
\newtheorem{proposition}[theorem]{Proposition}

\theoremstyle{definition}
\newtheorem{definition}[theorem]{Definition}
\newtheorem{example}[theorem]{Example}

\newtheorem{problem}[theorem]{Problem}

\theoremstyle{remark}
\newtheorem{remark}[theorem]{Remark}

\newcommand{\nilcox}{\mathbb{A}}
\newcommand{\nilcoxmon}{\hat{\mathbb{A}}}
\newcommand{\core}{\mathfrak{c}}
\newcommand{\sh}{\operatorname{sh}}
\newcommand{\len}{\operatorname{len}}
\newcommand{\ZZ}{\mathbb{Z}}

\newcommand{\LD}{\operatorname{LD}}
\newcommand{\LI}{\operatorname{LI}}
\newcommand{\RD}{\operatorname{RD}}
\newcommand{\RI}{\operatorname{RI}}

\newcommand{\LCRD}{\operatorname{CRD}}

\newcommand{\tp}{\operatorname{TOP}}
\newcommand{\dwn}{\operatorname{DOWN}}
\newcommand{\chrm}{\operatorname{CHARM}}

\newcommand{\AS}{\tilde{S}}

\usepackage{ifthen}
\newboolean{draft}
\setboolean{draft}{true}
\ifdraft
\newcommand{\TODO}[2][To do: ]{\textcolor{red}{\textbf{#1#2}}}
\else
\newcommand{\TODO}[2][]{}
\fi

\begin{document}
\title{\bf Canonical Decompositions of Affine Permutations, Affine Codes,\\ and Split $k$-Schur Functions}
\author{Tom Denton}

\author{Tom Denton\thanks{Supported by York University and the Fields Institute.}\\
\small Department of Mathematics\\[-0.8ex]
\small York University\\[-0.8ex] 
\small Toronto, Ontario, Canada\\
\small\tt sdenton4@gmail.com\\
}

\date{Oct 8, 2012}

\maketitle

{\abstract{We develop a new perspective on the unique maximal decomposition of an arbitrary affine permutation into a product of cyclically decreasing elements, implicit in work of Thomas Lam \cite{Lam06affinestanley}.  This decomposition is closely related to the affine code, which generalizes the $k$-bounded partition associated to Grassmannian elements.  We also prove that the affine code readily encodes a number of basic combinatorial properties of an affine permutation.  As an application, we prove a new special case of the Littlewood-Richardson Rule for $k$-Schur functions, using the canonical decomposition to control for which permutations appear in the expansion of the $k$-Schur function in noncommuting variables over the affine nil-Coxeter algebra.}}

\tableofcontents

\section{Introduction}

The affine permutation group $\AS_{k+1}$ was originally described by Lusztig~\cite{lusztig83} as a combinatorial realization of the affine Weyl group of type $A_k^{(1)}$.  The affine permutations have since been extensively studied; a very good overview of the basic results may be found in~\cite{Bjorner_Brenti.2005}.  The group $\AS_{k+1}$ is important for a variety of reasons; for example, new results on $\AS_{k+1}$ often generalize or give new results in the classical symmetric group.  Additionally, $\AS_{k+1}$ is the affine Weyl group of type $A_{k}^{(1)}$, and new combinatorics in the affine symmetric group suggest new directions of exploration for general affine Weyl groups.  Finally, the affine nil-Coxeter algebra, which is closely related to the affine permutation group, has proven very useful in the study of symmetric functions, via the construction of Schur (and $k$-Schur) functions in non-commuting variables~\cite{FominGreene98}\cite{Lam06affinestanley}.

As our primary objective, we develop new machinery for finding the unique maximal decomposition of an arbitrary affine permutation.  This may be interpreted as a canonical reduced decomposition of each affine permutation.  This composition is encoded in the \emph{affine code}, or \emph{$k$-code}, which may be interpreted as a weak composition with $k+1$ parts, at least one of which is zero\footnote{The affine code generalizes both the inversion vector and the Lehmer code of a classical permutation.}.  We interpret the diagram of an affine code as living on a cylinder.  (Before the connection to the affine code was noticed, we called this object a $k$-castle, because when the affine code's Ferrer's diagram is drawn on a cylinder, it resembles the ramparts of a castle.  The requirement that one part of the composition is zero means that the castle always has a ``gate.''  See Figure~\ref{fig.affineTableau}.)  One may then quickly determine whether two affine permutations given by reduced words are equal by putting each in their canonical form: Thus, we provide an alternative solution of the word problem for the affine symmetric group.  The affine code readily yields other useful information about the affine permutation, including its (right) descent set and length.  Furthermore, the Dynkin diagram automorphism on $\AS_{k+1}$ may be realized by simply rotating the affine code.

While in review, it was noticed that the $k$-code coincides with the affine code, and that the unique maximal decomposition is implicit in the work of Lam~\cite{Lam06affinestanley}[Theorem 13].  Our work here gives an alternative proof of the existence and uniqueness of the maximal decomposition, as well as introducing the perspective of the two-row moves on cyclic decompositions of an element.  This added perspective is imminently useful when considering combinatorial problems arising in the affine symmetric group; this utility is demonstrated in Section~\ref{sec:lwrule}, and will be further demonstrated in future work.

We furthermore describe an insertion algorithm on affine codes, which gives rise to the notion of a set of standard recording tableau in bijection with the set of reduced words for an affine permutation with affine code $\alpha$.  We also generalize a number of constructions that arise in the study of $k$-Schur functions (described below) to general affine permutations.  In particular, the notions of $k$-conjugation and weak strip appear and generalize naturally in the study of affine codes.

Initially we developed this machinery in order to prove a special case of the $k$-Littlewood-Richardson rule describing the multiplication of $k$-Schur functions.  The $k$-Schur functions $s_\lambda^{(k)}$ are indexed by $k$-bounded partitions, and give a basis for the ring $\Lambda^{(k)}$ defined as the algebraic span of the complete homogeneous functions $h_i$ with $i\leq k$.

The $k$-Schur functions were originally defined combinatorially in terms of $k$-atoms, and conjecturally provide a positive decomposition of the Macdonald polynomials~\cite{LLM2003}.  Since their original appearance, these functions have attracted much attention, but many basic properties remain elusive.  As of this writing, the author estimates that there are at least five different definitions, all of which are conjecturally equivalent.  A good overview of the state-of-the-art in the study of $k$-Schur functions, including many of the various definitions, is~\cite{kschurprimer}.

One definition of the $k$-Schur functions is given by the $k$-Pieri rule.  The $k$-bounded partitions are are in bijection with $(k+1)$-cores and Grassmannian affine permutations.    Lam demonstrated that the cyclically decreasing elements in the affine nil-Coxeter algebra commute and satisfy the same multiplication as the $h_i$'s~\cite{Lam06affinestanley}.  As such, the $k$-Pieri rule may be used to construct elements in the affine nil-Coxeter algebra which mimic the $k$-Schur functions.  This is the realization of the $k$-Schur functions we use throughout this paper.

\begin{definition}
\label{def:splits}
Given a shape $\mu$, the \emph{$k$-boundary} $\partial_k(\mu)$ of $\mu$ is the skew shape obtained by removing all boxes with hook $>k$.
A skew shape is \emph{connected} if any box may be reached from any other box by a sequence of vertical and horizontal steps.
A $k+1$-core $\mu$ \emph{splits} if the $k$-boundary $\partial_k(\mu)$ is not connected.  If $\mu$ splits, then each connected component of $\partial_k(\mu)$ is the boundary of some $k+1$ core $\rho_i$.  These cores $\rho_i$ are the \emph{components} of $\mu$.  Any collection of diagonally-stacked connected components may similarly be associated to a core; such a collection we call a \emph{factor}, in anticipation of the main result.
\end{definition}

Our main application is the following special case of the $k$-Littlewood-Richardson rule, which appears as Theorem~\ref{thm:mainResult}:
\begin{theorem}
Suppose $\mu$ splits into components $\mu_i$.  Then 
\[
s_\mu^{(k)} = \prod s_{\mu_i}^{(k)}.
\]
\end{theorem}

\begin{example}
Consider the $5$-core $(6,3,3,1,1,1)$, associated to the $4$-bounded partition $(3,2,2,1,1,1)$:

  \begin{center}
  \includegraphics[scale=1]{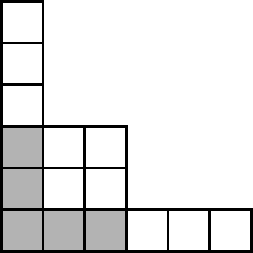}
  \end{center}

The $5$-boundary is in white, while the non-boundary boxes are shaded grey.  The boundary splits into three connected components, $(1,1,1)$, $(2,2)$, and $(3)$.  Then the theorem states that:
\[
s_{(6,3,3,1,1,1)}^{(4)} = s_{(3)}^{(4)} s_{(2,2)}^{(4)} s_{(1,1,1)}^{(4)}.
\]
\end{example}

This special case of the $k$-Schur Littlewood Richardson rule is similar in flavor to one proven by Lapointe and Morse~\cite{lapointeMorse2007}.  Their special case, the $k$-rectangle rule, involves multiplication of a $k$-Schur function indexed by a rectangle $R$ with maximal hook $(k+1)$ by an arbitrary $k$-Schur function $\lambda$.  In this case:
\[
s_R^{(k)}s_\lambda^{(k)} = s_{R\cup \lambda}^{(k)},
\]
where $R\cup \lambda$ is the partition obtained by stacking the Ferrer's diagrams of $R$ and $\lambda$ and then ``down-justifying''\footnote{Or ``up-justifying,'' if you prefer the English notation for partitions.} the resulting shape to obtain a $k$ bounded partition.  
Given the $k$-rectangle rule, the multiplication of the $k$-Schur functions for a fixed $k$ is then fully determined by the multiplication of the $k$-Schur functions indexed by shapes strictly contained within a $k$-rectangle.  The $k$-rectangle rule was given a combinatorial interpretation in the affine nil-Coxeter setting in~\cite{bbtz2011}.

The splitting condition we consider here is distinct from the $k$-rectangle rule, and provides some products of $k$-Schur functions contained strictly within a $k$-rectangle, and thus advances our overall understanding of the $k$-Littlewood-Richardson rule.

\subsection{Further Directions.}

Our results on affine codes suggest a number of questions for further exploration.  In particular, we expect that our perspective will be helpful in problems relating to reduced decompositions of affine permutations, especially those relating to the affine Stanley symmetric functions $\mathfrak{S}_x$, originally studied in~\cite{Lam06affinestanley}.  The affine Stanley symmetric function may be defined as a sum over decompositions of an affine permutation into a product of cyclically decreasing elements; our framework gives a natural way to relate these various decompositions, which we will explore in further work.  The affine codes may also be helpful in the enumeration of reduced words for either classical or affine permutations, a problem which has proven especially difficult.  

As noted in~\cite{Lam06affinestanley}, the problem of expanding the $k$-Schur functions over the nil-Coxeter algebra $\nilcox_k$ is equivalent to finding the $k$-Littlewood-Richardson coefficients.  A number of the supporting results in this work determine coefficients in the expansion of the $k$-Schur function for special elements, using the affine code constructions.  We expect that more information about the coefficients in the expansion may be gleaned from further study, which in turn will illuminate the $k$-Littlewood-Richardson coefficients.

\subsection{Overview}
In Section~\ref{sec:background} we review basic concepts from the literature and establish notation that will be used throughout the paper.  This includes a review of affine permutations, the affine nil-Coxeter algebra, cyclically decreasing elements, and the expression of the $k$-Schur functions in non-commuting variables over the affine nil-Coxeter algebra..

The bulk of the paper is in Section~\ref{sec:kcastles}.  In this section, we construct the bijection(s) between affine permutations and $k$-codes, via maximal decompositions.
In Subsection~\ref{ssec:maxCyclic}, we prove that every affine permutation has a unique maximal decomposition as a product of cyclically decreasing elements.  This provides the first main result of the paper, Theorem~\ref{thm:maxDecomposition}.  The proof of the theorem is constructive, and provides a fast algorithm for computing the maximal decomposition.

In Subsection~\ref{ssec:maxMoves}, we establish `moves' between various reduced decompositions of an affine permutation into cyclically decreasing elements.  These allow us to prove Proposition~\ref{prop:maxCyclic}, which establishes that the maximal decomposition of any affine permutation into cyclically decreasing elements satisfies a `shifted containment' property, which is key in the identification of the decomposition with a weak composition.

In Subsection~\ref{ssec:insertion}, we reinterpret the maximizing moves to establish an insertion algorithm on $k$-codes.  This algorithm is reversible, which allows us to associate a set of recording tableaux to each $k$-code, and thus to each affine permutation.  By construction, these recording tableaux are in bijection with reduced words for the affine permutation.  This is the content of Theorem~\ref{thm:redWords}.

In Subsection~\ref{ssec:bijection}, we prove our main result, Theorem~\ref{thm:compTabBijection}, which establishes the bijection between $k$-codes and affine permutations.  We also establish a relationship between descents of $k$-codes and descents of the affine permutations they correspond to.

In Subsection~\ref{ssec:relations}, it is observed that there are actually four different bijections between affine permutations and $k$-codes, according to different choices for the maximal decomposition:  One can build either a decomposition into cyclically decreasing or increasing elements, from the right side or the left side.  Here we investigate the relationships between the four $k$-codes assigned to a given affine permutation.  The increasing and decreasing decompositions are related by a generalization of the $k$-conjugate, a vital construction on $k$-bounded partitions. 
We also note that the $k$-codes of the left and right decreasing decompositions are related by a permutation (Proposition~\ref{prop:colPermutation}).

We then focus on Grassmannian elements in Subsection~\ref{ssec:grassmannian}.  These are affine permutations with right descent set $\{0\}$ or $\emptyset$.  They are of particular interest because they index the $k$-Schur functions: Grassmannian elements are in bijection with $k$-bounded partitions, which may be interpreted as a $k$-code $\alpha$ with only one descent at $\alpha_0$.  We show that the usual $k$-conjugate of $k$-bounded partitions corresponds to switching between two maximal decompositions of the associated Grassmannian element (Proposition~\ref{prop:kconjagrees}).  This allows us to define the $k$-conjugate on arbitrary affine permutations.

The $k$-Pieri rule is used to define the $k$-Schur functions, and an important characterization of the Pieri rule is by weak horizontal strips.  In particular, consider $k$-bounded partitions $\lambda\subset \mu$, and let the $k$-conjugate of $\lambda$ and $\mu$ be $\lambda'$ and $\mu'$ respectively.  (The $k$-conjugate is defined in Section~\ref{ssec:kSchurFunctions}.)  Then we say that the skew shape $\mu/\lambda$ is a \emph{weak strip} if no column of $\mu/\lambda$ contains two boxes, and no row of $\mu'/\lambda'$ contains two boxes.  Suppose the affine permutations associated to $\lambda$ and $\mu$ are $x$ and $y$ respectively.  Indeed, $\mu/\lambda$ is a weak strip if and only if there exists a cyclically decreasing element $d_A$ such that $y=d_Ax$.  

In Subsection~\ref{ssec:genpieri}, we generalize the combinatorial Pieri rule by showing that multiplying any affine permutation by a cyclically decreasing element adds at most one box to each row of its $k$-code, while multiplication by a cyclically increasing element adds at most one box to each column of its $k$-code.

Section~\ref{sec:incrdecr} investigates the results of multiplying cyclically increasing and decreasing elements together.  In particular, we find near-commutation rules: for cyclically increasing and decreasing elements $u_B$ and $d_A$, there exist $A', B'$ such that $u_Bd_A=d_{A'}u_{B'}$.  The main result is Proposition~\ref{prop:updowndownup}.

Finally, we use the machinery of the previous sections to prove our special case of the $k$-Littlewood Richardson rule in Section~\ref{sec:lwrule}.  The main result is Theorem~\ref{thm:mainResult}.

\subsection{Acknowledgements}

I would like to thank the Fields Institute for providing space and hosting our weekly Algebraic Combinatorics seminar, where many of the ideas in this paper were discussed, argued, and strengthened.  Many thanks are also due to Chris Berg, Nantel Bergeron, Zhi Chen, Anne Schilling, Luis Serrano, Nicolas Thi\'ery and Mike Zabrocki for helpful comments and conversation.  Funding during this research was provided by Nantel Bergeron and York University.  Thanks are also due to the LACIM group at UQAM for their repeated hospitality.

One reviewer of the article patiently indicated many small errors in the text and supplied many generally helpful comments, for which I am grateful.  Of course, any remaining errors are my own.

Finally, much of the work in this project made extensive use of the Sage computer algebra system and the Sage-Combinat project~\cite{Sage}\cite{Sage-Combinat}, and was particularly reliant on code contributions from Chris Berg, Franco Saliola, Anne Schilling, and Mike Zabrocki.

\begin{figure}
  \begin{center}
  \includegraphics[scale=0.75]{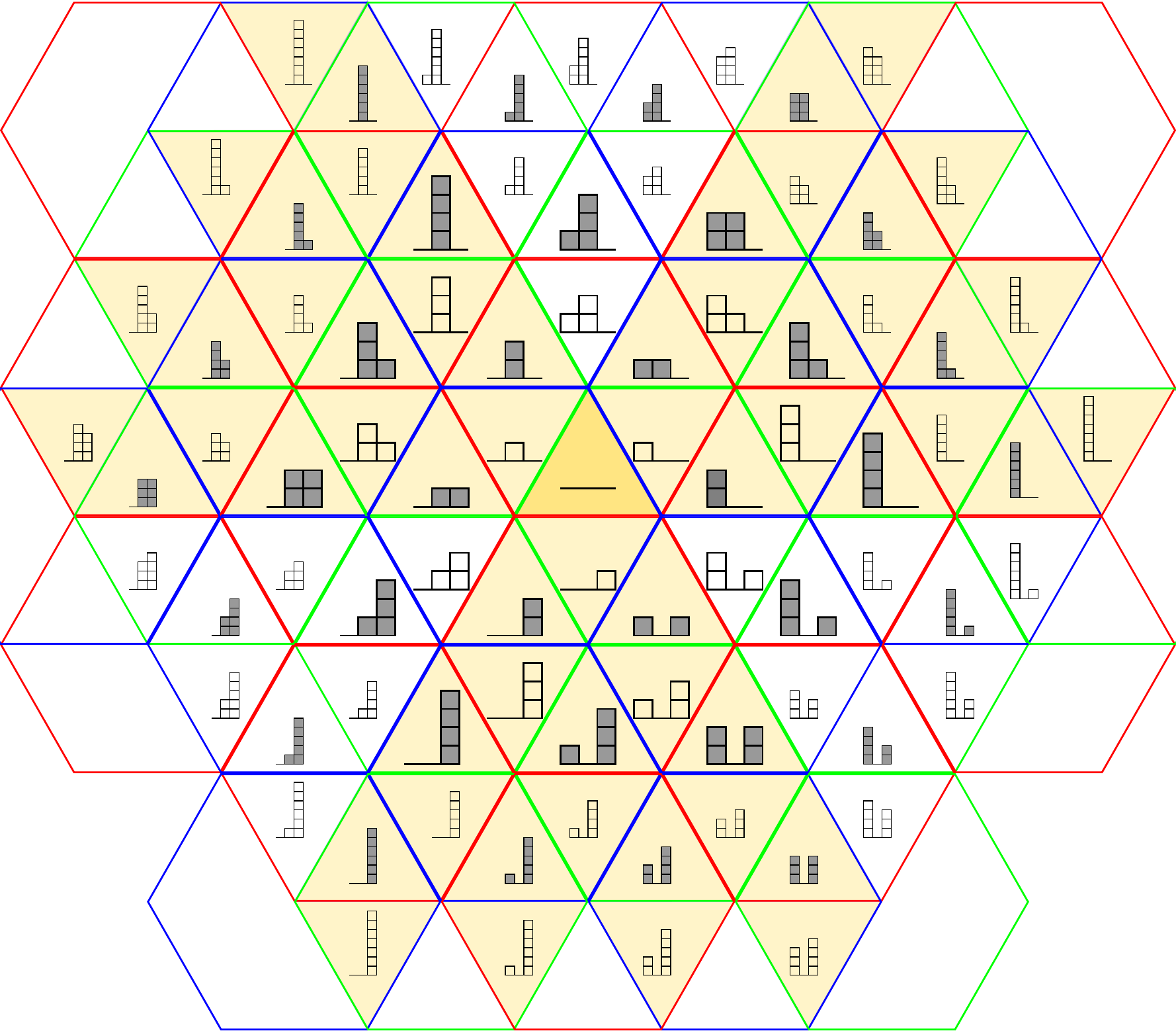}
  \end{center}
  \caption{All $2$-castles with $\leq 7$ boxes, drawn in the alcoves corresponding to the various affine permutations in $\AS_3$.  The colored walls of the alcoves indicate which simple transposition is used to cross that wall.  Blue is $s_0$, green is $s_1$, and red is $s_2$.  The orange-shaded regions indicate the three dominant cones; in particular, the $0$-dominant elements are in the cone which opens to the upper-right.  The shading in the boxes of the castles themselves corresponds to length; odd-length permutations have no shading, while even-length permutations are shaded gray.}
  \label{fig.2-castles-5}
\end{figure}

\begin{figure}
  \begin{center}
  \includegraphics[scale=0.75]{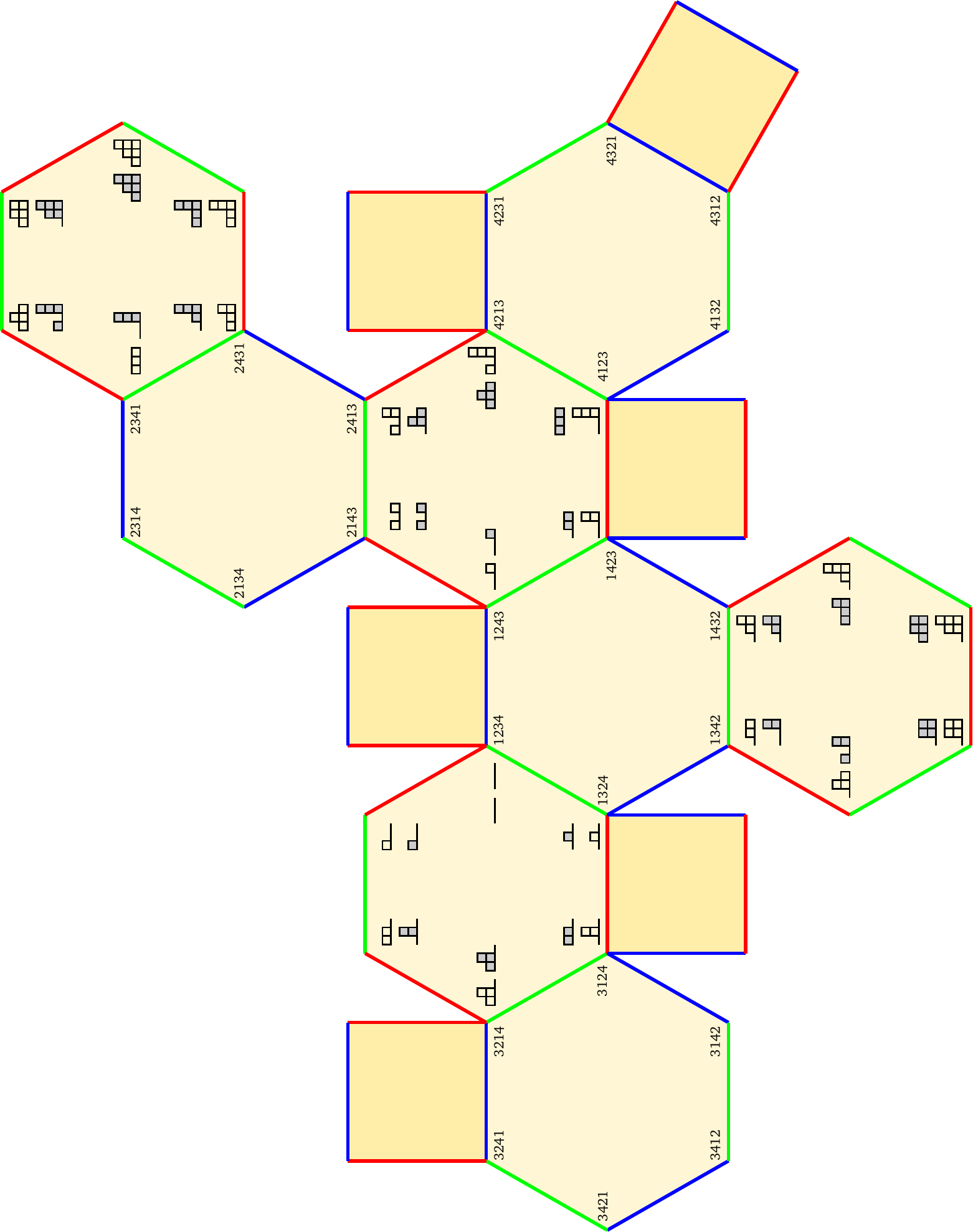}
  \end{center}
  \caption{DIY Permutahedron:  Cut out and glue together to get a 3-dimensional representation of this paper's results, restricted to the classical symmetric group $S_4$.  For best results, print large and in color on heavy card-stock.  Each vertex is labeled by a permutation, and the $k$-codes associated to the right decreasing (in white) and increasing (gray) decompositions of the permutation.  Edges correspond to right multiplication by a simple transposition: red for $s_1$, green for $s_2$ and blue for $s_3$.  You can use the blank square faces to mark your own favorite permutation statistics!}
  \label{fig.permutahedronModel}
\end{figure}

\section{Background and Definitions}
\label{sec:background}

In this section we review background material and fix notations for the remainder of the paper.

\subsection{The Affine Nilcoxeter Algebra and Affine Permutations.}

We begin by defining the affine nil-Coxeter algebra, and reviewing some basic facts and definitions relating to affine permutations.  Good references on affine Coxeter groups in general and the affine symmetric group in particular include~\cite{humphreys90},
\cite{Bjorner_Brenti.2005}.

Let $k$ be a positive integer.  Let $I$ indicate the index set $\ZZ_{k+1}=\{0, 1, \ldots, k\}$, which correspond to nodes in the Dynkin diagram of type $A_k^{(1)}$.  Indices from $I$ are thus always considered modulo $k+1$.  The \emph{Dynkin diagram} of type $A_k^{(1)}$ is the cyclic graph with vertices labeled by elements of $I$, and an edge connecting each pair of indices $i$ and $i+1$.
For brevity, we let $[p, q] := \{p, p+1, \ldots, q-1, q\}\subsetneq I$ for $p\neq q-1$.  (For example, with $k=5$, the set $[4,2]=\{4,5,0,1,2\}$.)  We call a subset $A\subsetneq I$ \emph{connected} if the corresponding subgraph of the Dynkin diagram is connected; \emph{i.e.,} $A=[i,l]$ for some $i, l$.  A connected component of an arbitrary $A\subsetneq I$ is a maximal connected subset of $A$.

\begin{definition}
The \emph{affine nil-Coxeter monoid} $\nilcoxmon_k$ is generated by the alphabet $\{a_i \mid i\in I\}$, subject to the relations:
\begin{itemize}
\item $a_i^2=0$,
\item $a_ia_j=a_ja_i$ for all $j>i$ with $j-i>1$, 
\item $a_ia_{i+1}a_i=a_{i+1}a_ia_{i+1}$ for all $i$, and
\item $x0=0x=0$ for all $x\in \nilcoxmon_k$.
\end{itemize}
The affine nil-Coxeter algebra $\nilcox_k$ is the monoid algebra of $\nilcoxmon_k$.  The \emph{classical nil-Coxeter monoid} $\nilcoxmon_k^{0}$ and corresponding monoid algebra $\nilcox_k^{0}$ are obtained as a (parabolic) quotient of $\nilcox_k$ by evaluating $a_0=0$.

We will also occasionally use the \emph{affine symmetric group} $\AS_k$ generated by the alphabet $\{s_i \mid i\in I\}$, subject to the relations:
\begin{itemize}
\item $s_i^2=1$,
\item $s_is_j=s_js_i$ for all $j>i$ with $j-i>1$, and
\item $s_is_{i+1}s_i=s_{i+1}s_is_{i+1}$ for all $i$.
\end{itemize}
\end{definition}

Elements of the affine symmetric group may be considered as (affine) permutations $x: \ZZ\rightarrow \ZZ$ subject to the additional requirements that:
\begin{itemize}
\item $\sum_{i=1}^{k+1} x(i) = \binom{k+2}{2} $, and
\item $x(i+k+1) = x(i)+k+1$.
\end{itemize}
Any affine permutation $x$ is then completely specified by its \emph{window notation}, given by the vector $(x(1), x(2), \ldots, x(k+1))$.  Affine permutations may be considered as bi-infinite sequences, setting $x_l:=x(l)$.
These affine permutations are in bijection with non-zero elements of the nil-Coxeter monoid.  

The generators $s_i$ may be considered as the simple transpositions exchanging $m(k+1)+i$ with $m(k+1)+i+1$ for every $m\in \ZZ$.  The set of affine permutations admit a \emph{left action} and a \emph{right action} by the generators $s_i$.  Considering $x$ as a bi-infinite sequence $(\ldots, x(-1), x(0), x(1), \ldots)$, we may consider the left action of the generators as an action on values (exchanging $m(k+1)+i$ with $m(k+1)+i+1$ for every $m\in \ZZ$), while the right action is on positions (exchanging the values in positions $m(k+1)+i$ and $m(k+1)+i+1$ for every $m\in \ZZ$).    

A \emph{reduced word} or \emph{reduced expression} for $x$ is minimal length sequence $(w_1, \ldots, w_l)$ with $w_l\in I$ such that $x=s_{w_1}\cdots s_{w_l}$.  The number $l$ is the \emph{length} of $x$, which we denote $\len(X)$.  It is a consequence of basic Coxeter theory that an expression is reduced if and only if $a_{w_1}\cdots a_{w_l}\neq 0$.  We mainly consider affine permutations as elements of the nil-Coxeter monoid, partially because this is the natural setting to work in for the $k$-Schur functions, and partially to avoid worrying about whether a given expression is reduced.

To save space, we will often write words in $\nilcox_k$ as a subscript: for example, we write $a_1a_2a_1$ as $a_{1,2,1}$.

Let $x$ be an affine permutation.  We recall the \emph{set of right descents} $D_R(x)\subsetneq I$ of an element $x$.  We say that $x$ has a \emph{right descent} at $i\in I$, and write $i\in D_R(x)$,
\begin{itemize}
\item $x(i)>x(i+1)$, 
\item $x$ has a reduced word ending with the generator $s_i$.
\end{itemize}
Likewise, we define the left descents $D_L(x)\subsetneq I$.  Recall that $x$ has a \emph{left descent} at $i\in I$, and $i\in D_L(x)$ if either of the following two equivalent statements hold:
\begin{itemize}
\item $i$ appears to the right of $i+1$ in $x$ considered as a bi-infinite sequence,
\item $x$ has a reduced word beginning with the generator $s_i$.
\end{itemize}
Note that for any $x$, $D_R(x)\neq I$.  (If $D_R(x)=I$, then $x$ would be a longest element in $\AS_k$.  But such elements do not exist in affine Coxeter groups for a variety of reasons.~\cite{humphreys90})

In Figure~\ref{fig.2-castles-5}, illustrating the bijection between $2$-castles and affine permutations in $\AS_3$, we make use of the \emph{alcove model} for affine permutations, for which we refer the unfamiliar reader to~\cite{humphreys90}.  In short, each triangle in the picture is an `alcove,' corresponding to a particular affine permutation.  Crossing a wall of an alcove to reach an adjacent alcove corresponds to multiplication by a simple transposition.

\begin{lemma}[Extended Braid Relation]
\label{lem:extendedBraidRelation}
For any set $[i,j] \subsetneq I$, we have:
\[
a_{i, i+1, \ldots, j-1, j, j-1, \ldots, i+1, i} = a_{j, j-1, \ldots, i+1, i, i+1, \ldots, j-1, j}.
\]
\end{lemma}
\begin{proof}
This follows from repeated application of the braid relation.
\end{proof}

The Dynkin diagram of type $A_k^{(1)}$ admits a cyclic symmetry, which descends to an algebra automorphism of $\nilcox_k$.
\begin{definition}
The Dynkin Diagram automorphism $\Psi: \nilcox_k \rightarrow \nilcox_k$ is defined by its action on the generators:
\[
\Psi(a_i) = a_{i+1}.
\]
\end{definition}
We observe that $\Psi^{(k+1)}$ is the identity.

\subsection{Cyclic Elements in $\nilcox_k$.}

\begin{definition}
Given a subset $A \subsetneq I$ with $|A|=n$ we define the cyclically decreasing element $d_A$ ($d$ for `down') to be the product $d_A:=a_{i_1}\cdots a_{i_n}$ for $i_l\in A$, where if $j, j-1\in A$ then $j$ appears to the left of $j-1$ in any reduced word for $d_A$.  The cyclically increasing element $u_A$ ($u$ for `up') is defined similarly, where if $j$ and $j-1\in A$ then $j$ appears to the right of $j-1$ in any reduced word for $d_A$.  

Then we define:
\[
h_i := \sum_{|A|=i} d_A, \text{ and }
e_i := \sum_{|A|=i} u_A.
\]
For a partition $\lambda = (\lambda_1, \ldots, \lambda_n)$, $h_\lambda= \prod h_{\lambda_i}$, and $e_\lambda= \prod e_{\lambda_i}$.

We frequently use the notation $A-1:= \{i-1 \mid i \in A\}$ and occasionally $A+1:=\{i+1 \mid i \in A\}$.

A \emph{cyclically increasing} (respectively, \emph{cyclically decreasing}) element of $\nilcox_k$ is an element specified by ordered collection of subsets $A_i\subsetneq I$, given by $u_{A_n}u_{A_{n-1}} \cdots u_{A_1}$ (respectively, $d_{A_n}d_{A_{n-1}} \cdots d_{A_1}$).  We abbreviate such products using the notation $\vec{A} := \{A_1, \ldots, A_n\}$, so that $u_{\vec{A}}:=u_{A_n} \cdots u_{A_1}$ and $d_{\vec{A}}:=d_{A_n} \cdots d_{A_1}$.  A cyclically increasing (resp. decreasing) product $x$ is \emph{maximal} if the \emph{shape of $\vec{A}$} given by the vector $\sh(A) := (|A_1|, |A_2|, \ldots, |A_n|)$ is lexicographically maximal amongst all cyclically increasing (resp. decreasing) expressions for $x$.  
\end{definition}

\begin{example}
Let $k=5$, so that $I=\{0,1,2,3,4,5\}$.  Set $A=\{0,2,4,5\}$.  Then $d_A=a_0a_5a_4a_2$, and $u_A=a_2a_4a_5a_0$.  There is a bijection between proper subsets of $I$ and cyclically decreasing elements.
\end{example}

\begin{theorem}[\cite{Lam06affinestanley}]
The elements $h_i$ generate a commutative subalgebra of $\nilcox_k$.
\end{theorem}

\begin{definition}
The \emph{right descent set} of an element $w\in \nilcox_k$ is the set $D_R(w) := \{p\in I \mid w s_p = 0 \}$.  The left descent set $D_L(w)$ is defined similarly.

For $p\in I$, an element $w\in \nilcox_k$ is \emph{$p$-dominant} if $D_R(w)\subset \{p\}$.  When $p=0$, such elements are also known as \emph{Grassmannian elements}.
\end{definition}

\begin{lemma}
A cyclically decreasing (or increasing) element is connected if and only if it is $i$-dominant for some $i$.  
\end{lemma}
\begin{proof}
If not connected, then the element has multiple descents.  If it is connected, no relations may be applied to the element, and so there is only one right descent.
\end{proof}

\subsection{$k$-Schur Functions.}
\label{ssec:kSchurFunctions}

The literature on $k$-Schur functions is extensive, but an excellent overview is given in ``A Primer on $k$-Schur Functions,'' by Schilling and Zabrocki~\cite{kschurprimer}.  Additional background on the realization of the $k$-Schur functions in non-commuting variables over the affine nil-Coxeter algebra may be found in~\cite{bbtz2011}.

\begin{definition}
A $k$-bounded partition is a partition $\lambda=(\lambda_1, \ldots, \lambda_n)$ with each $\lambda_i\leq k$.
A $k+1$-core is a partition $\mu$ with no hooks of length $k+1$.  Given a $k+1$-core $\mu$, the \emph{$k+1$-boundary} $\partial_{k+1}(\mu)$ is the skew shape obtained by deleting all boxes of $\mu$ with hook length greater than $k+1$.  When $k$ is not ambiguous, we will just write $\partial(\mu)$.
\end{definition}

There is a well-known bijection between $k$-bounded partitions and $k+1$ cores.  The bijection is defined by an algorithm on the bounded partition: starting with the first row of the Ferrer's diagram for $\lambda$, if the first box $b$ of a given row has hook length $\geq k+1$, we add boxes to the beginning of the row until the box $b$ has hook length $\leq k$.  We perform this operation on each row of $\lambda$ sequentially to obtain a $k+1$-core.  We may recover the $k$-bounded partition by taking the $k$-boundary of the $k+1$-core and pushing all of the boxes in the resulting skew shape to the left to form a partition.

For a $k$-bounded partition $\lambda$, we write $\mathfrak{c}(\lambda)$ for the associated $k+1$-core, and for a $k+1$-core $\mu$, write $\mathfrak{p}(\mu)$ for the associated $k$-bounded partition.  Thus $\mathfrak{p}(\mathfrak{c}(\lambda))=\lambda$.

Of considerable importance in the study of $k$-Schur functions is the $k$-conjugate.  For a $k$-bounded partition $\lambda$, the $k$-conjugate $\lambda^{(k)}:= \mathfrak{p}(\mathfrak{c}(\lambda)')$, where $\mathfrak{c}(\lambda)'$ denotes the usual conjugate of the $k+1$-core associated to $\lambda$.  Notice that the usual conjugate of a $k$-bounded partition need not be $k$-bounded; the $k$-conjugate returns another $k$-bounded partition.  (See Figure~\ref{fig.boundedCoreBiject}.)  The $k$-conjugate combinatorially implements on the level of $k$-Schur functions the automorphism of the symmetric functions which exchanges $h_i$ and $e_i$.

\begin{figure}
  \begin{center}
  \includegraphics[scale=0.75]{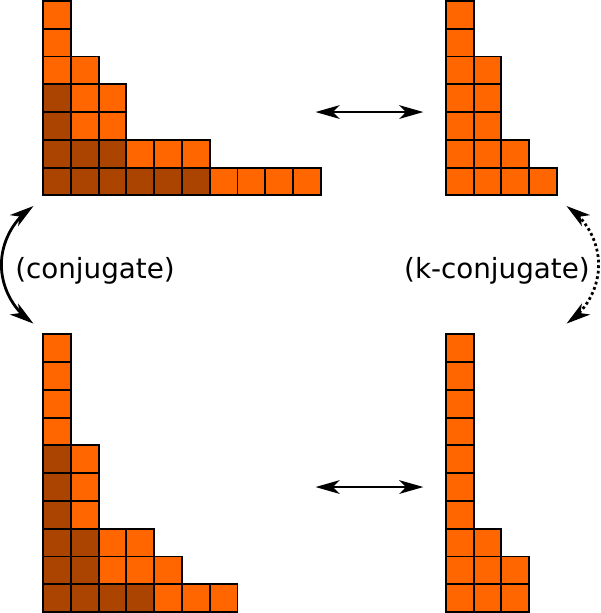}
  \end{center}
  \caption{Illustration of the bijection between $k$-bounded partitions and $k+1$-cores and the $k$-conjugation operation.  In this case, $k=4$.  On the left are two $5$-cores with the boxes with hook $>5$ shaded darker, and on the right are the corresponding $4$-bounded partitions.  The $5$-cores are related by conjugation; the $4$-bounded partitions are related by $k$-conjugation.}
  \label{fig.boundedCoreBiject}
\end{figure}

Recall that for $\lambda, \nu$ partitions of $n$, $\nu$ \emph{dominates} $\lambda$ if $0\leq \sum_{i=1}^j (\nu_i-\lambda_i)$ for every $j$ (possibly padding one of the partitions with zeroes if their lengths are unequal).  In this case, we write $\nu\succ \lambda$.

The $k$-Schur functions are indexed by $k$-bounded partitions, and may be defined by the Pieri rule.  The Pieri rule gives an inductive definition of the $k$-Schur functions, by setting $s_{(l)}^{(k)}:=h_l$, and then expressing $h_l s_\mu^{(k)}=\sum_\nu s_\nu^{(k)}$ according to some restrictions on $\nu$.  In particular, the partitions $\nu$ satisfy a triangularity property with respect to the dominance order, allowing recursive definition of the $k$-Schur functions.

There are different interpretations of the Pieri rule in different contexts, but the primary definition is by weak horizontal strips.  Given partitions $\lambda \subset \mu$, we say that the skew shape $\mu\setminus\lambda$ is a \emph{horizontal strip} if each column of $\mu\setminus\lambda$ contains at most one box.  Likewise, it is a \emph{vertical strip} if each row contains at most one box.  If $\lambda \subset \mu$ are $k$-bounded partitions, we say that the skew shape $\mu\setminus\lambda$ is a \emph{weak horizontal strip} if $\mu\setminus \lambda$ is a horizontal strip and $\mu^{(k)}\setminus\lambda^{(k)}$ is a vertical strip.  Then the Pieri rule may be stated as:
\[
h_l s_\mu^{(k)} = \sum_\nu s_\nu^{(k)},
\] 
where $\nu\setminus\mu$ is a weak horizontal strip~\cite{lapointeMorse2003}.  As a consequence, we can observe that if $l$ is less than the last part of $\mu$, then
\[
h_l s_\mu^{(k)} = s_{\mu\cup l}^{(k)} + \sum_{\nu} s_\nu^{(k)},
\]
where each $\nu\setminus \mu$ is a vertical strip.  Furthermore, each $\nu$ dominates $\mu \cup l$.

Recall that there is a bijection between Grassmannian (or $0$-dominant) affine permutations in $\AS_k$ and $k$-bounded partitions.  Their relation to the $k$-Schur functions is described by the following theorem, which arises as a consequence of the Pieri rule:
\begin{proposition}
For $l \in \{1, \ldots, k\}$, $s_{(l)}^{(k)}:=h_l$.  Each $k$-Schur function $s_\lambda^{(k)}$ appears with multiplicity one in $h_\lambda$.  Furthermore, in its expansion in $\nilcox_k$, $s_\lambda^{(k)}$ contains a unique $0$-dominant summand, $w_\lambda$.
\end{proposition}

There is a second interpretation of the Pieri rule in the context of the affine nil-Coxeter algebra.  Take $x$ to be a $0$-dominant element of the affine nil-Coxeter monoid.  Then:  
\[
h_l s_x^{(k)} = \sum_y s_y^{(k)},
\]
where the sum is over Grassmannian elements $y$ such that $y=d_Ax$ for some $A\subsetneq I$ with $|A|=l$.

\begin{corollary}
\label{cor:shiftDominance}
Each $k$-Schur function $s_\lambda^{(k)}$ contains a unique $i$-dominant summand for each $i\in I$.
\end{corollary}
\begin{proof}
The statement is true for $i=0$.  One may obtain an $i$-dominant summand in $s_\lambda^{(k)}$ by applying $\Psi^i(w\lambda)$.  This summand is unique, or else we could apply $\Psi^{k-i}$ to obtain more than one $0$-dominant summand.
\end{proof}

An important part of our later proofs in this paper will rely on finding coefficients of certain elements in the expansion of $h_\lambda$, $e_\lambda$, or $s_\lambda^{(k)}$.  For this, we employ the notation:
\[
[x]f := \text{ coefficient of $x$ in $f$.}
\]
For example, in $h_l$, we have:
\[
[d_A]h_l = \delta_{|A|, l}.
\]

\section{Canonical Cyclic Decompositions and $k$-Codes}
\label{sec:kcastles}

We first consider products of cyclically decreasing elements.  All of the results in this section may be adapted to products of cyclically increasing elements with small modifications.  For example, results concerning $k$-bounded partitions for products of cyclically decreasing elements become statements about $k$-column bounded partitions for cyclically increasing products.  These different decompositions are explored in Section~\ref{ssec:relations}.

Suppose we have a collection of subsets $\vec{A} = \{A_1, \ldots, A_n\}$ such that each $A_i\subsetneq I$.  Then we can form a cyclically decreasing product $d_{\vec{A}} = d_{A_n}\cdots d_{A_1}$.  Trivially every element $w\in \nilcox_k$ has an expression as a cyclically decreasing  product, by taking any reduced expression of length $n$ and considering the element as a product of $n$ cyclically decreasing elements of length $1$.  

Our primary goal for this section is to show that every affine permutation has a maximal expression as a product of cyclically decreasing elements, in the sense that the vector $(|A_1|, \ldots, |A_n|)$ is lexicographically maximal amongst all cyclically decreasing decompositions of $x$.  Given such a maximal decomposition, we may associate it with its $k$-code, defined immediately below.  We then show that $k$-codes are in bijection with affine permutations.  Along the way we will create algorithms analogous to jeu de taquin and insertion on $k$-codes, corresponding respectively to maximizing a cyclically decreasing decomposition and multiplying by a single generator.

\begin{definition}
A \emph{$k$-code} is a function $\alpha: \ZZ_{k+1}\rightarrow \ZZ_{\geq 0}$ such that there exists at least one $i$ with $\alpha(i)=0$.  The \emph{window notation} for $\alpha$ is the vector $[\alpha_0, \alpha_{1}, \ldots, \alpha_{k-1}, \alpha_{k}]$.  We usually identify $\alpha$ with its window notation.

The \emph{diagram of a $k$-code} $\alpha$ is a Ferrer's diagram on a cylinder with $k+1$ columns, indexed by $\ZZ_{k+1}$ where the $i$-th column contains $\alpha_i$ boxes.  A \emph{$k$-code filling} is a marking of the diagram of $\alpha$ with residues from $\ZZ_{k+1}$, with the box in the $i$th column and $j$th row marked with residue $i-j$.  We may \emph{flatten} a $k$-code's diagram by cutting out a column $j$ with $\alpha_j=0$.  A \emph{reading word} of a $k$-code filling is obtained by reading the rows of this flattened filled diagram from right to left, beginning with the last row.

A \emph{non-maximal $k$-code filling} $S=\vec{A}$ is given by a collection of subsets $\{A_1, \ldots, A_n\}$ with $A_i\subsetneq I$.  The $i$th row of the diagram of $S$ contains the residues in $A_i$.

A \emph{$k$-column castle tableau} is defined similarly, but on a cylinder with $k+1$ rows marked with residues.  In this case, the flattening is obtained by removing a row $j$ with $\alpha_j=0$.  The reading word is obtained from a flattened tableau by reading the columns top-to-bottom, beginning with the right-most column.
\end{definition}

Because there is a unique $k$-code filling constructed from each $k$-code, we will commonly identify these two objects, referring to both as a $k$-code.  We will develop an insertion algorithm in Section~\ref{ssec:insertion}, which will produce two `tableaux'.  The first tableau is just the $k$-code filling, and the second is a `recording tableau,' which yields a new combinatorial object in bijection with cyclic decompositions of an affine permutation.  The $k$-shape filling may be considered as a tableau by analogy to the RSK algorithm, rather than being a chain in a poset.

Examples are provided in Figures~\ref{fig.affineTableau} and~\ref{fig.example0}.

The rows of a $k$-code filling each correspond to a cyclically decreasing element with the residues appearing in that row.  This cyclically decreasing element is invariant under different choices of flattening for the tableau; the reading words of flattened tableaux will be related by commutation relations in $\nilcox_k$.

Note that the number of boxes in a $k$-code filling is equal to the number of letters in the decomposition $d_{\vec{A}}$, providing a natural grading on $k$-codes which will correspond to the length grading on affine permutations.

We will show that $k$-codes are in bijection with affine permutations in Theorem~\ref{thm:compTabBijection}.

\begin{figure}
  \begin{center}
  \includegraphics[scale=0.5]{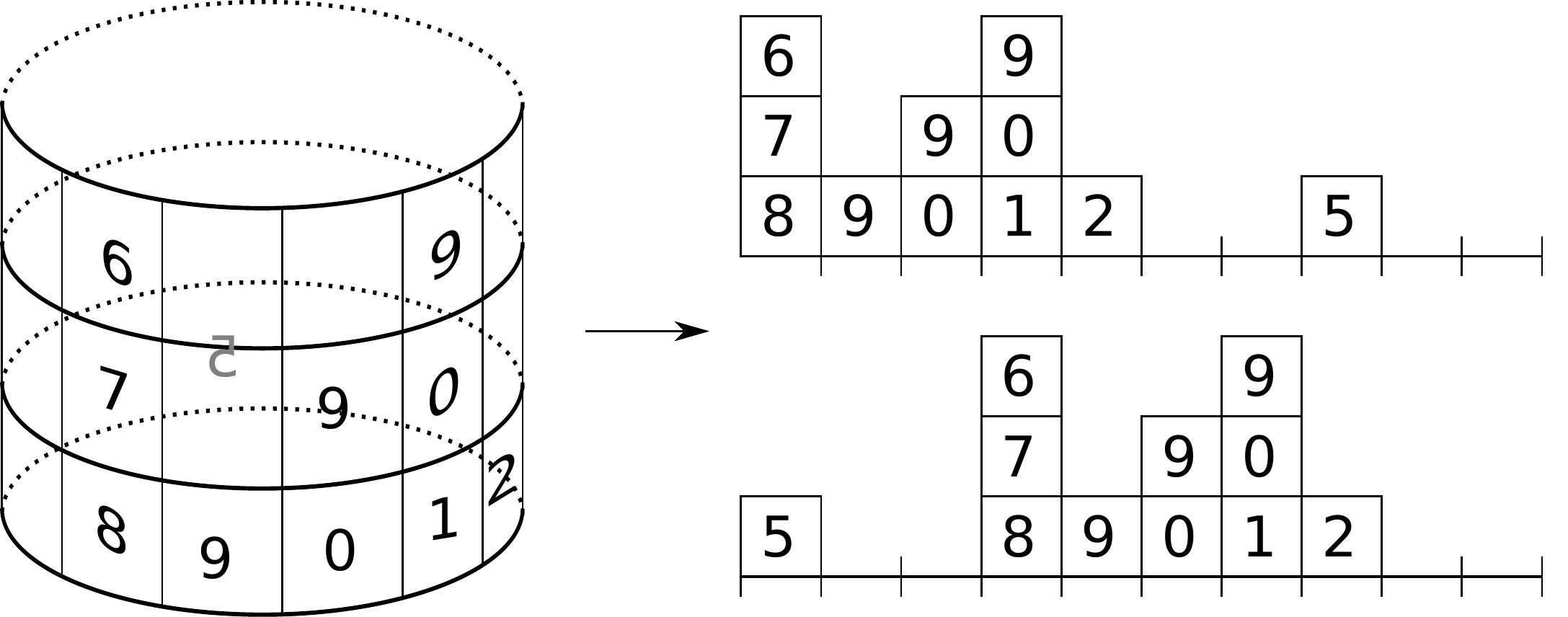}
  \end{center}
  \caption{On the left is the diagram of the $9$-core with window notation $(2,3,1,0,0,1,0,0,3,1)$ drawn on a cylinder.  On the right are two possible flattenings of this tableau.  The reading word of the top flattening is $(9,6, 0,9,7, 5,2,1,0,9,8)$.  The reading word of the bottom flattening is $(9,6, 0,9,7, 2,1,0,9,8,5)$.  These two reading words are related by commutation relations in $\nilcox_9$.}
  \label{fig.affineTableau}
\end{figure}

\begin{figure}
  \begin{center}
  \includegraphics[scale=0.75]{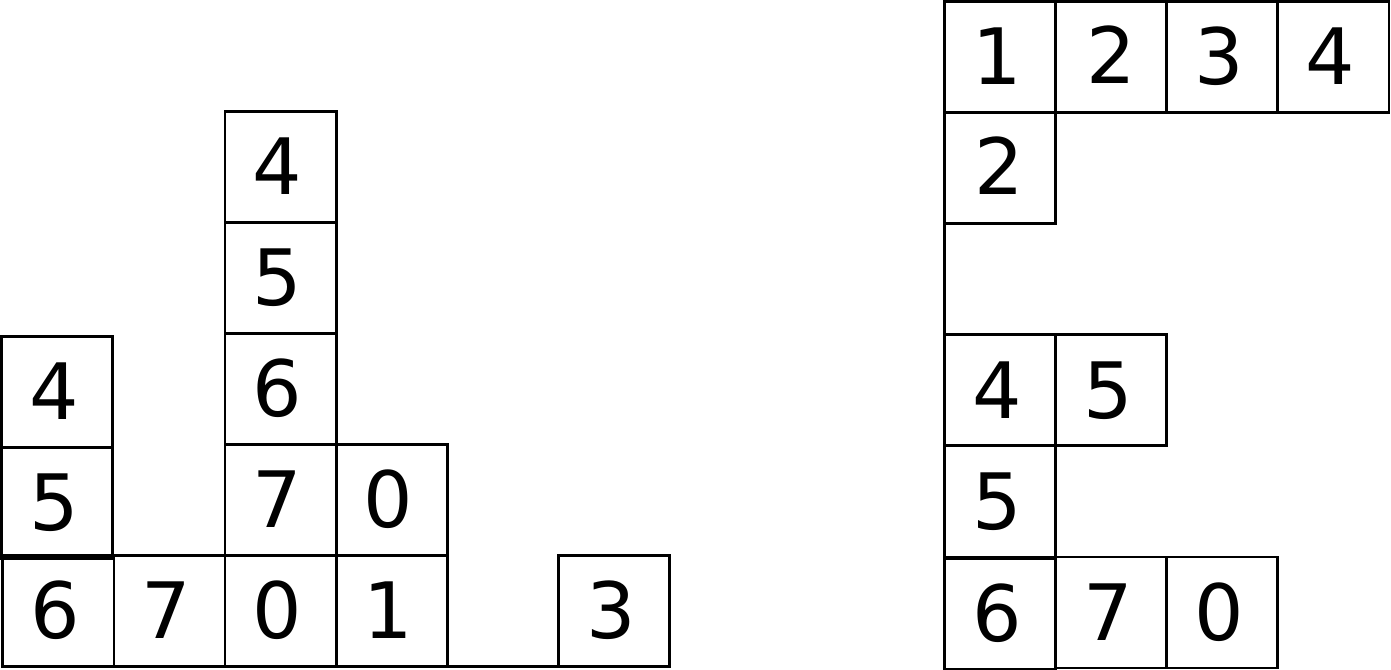}
  \end{center}
  \caption{On the left is a flattened $7$-castle tableau for the $7$-castle $\alpha=(5,2,0,1,0,0,3,1)$.  The reading word for this tableau is $(4,5,6,4,0,7,5,3,1,0,7,6)$.  On the right is a $7$-column bounded tableau for the $k$-code 
  $\rho=(0,4,1,0,2,1,3,0)$, whose reading word is $(4,3,0,2,5,7,1,2,4,5,6)$.}
  \label{fig.example0}
\end{figure}

\subsection{Maximal Cyclic Decompositions and $k$-Codes}
\label{ssec:maxCyclic}

Our first objective is to show that there exists a unique maximal set $A\subsetneq I$ such that $x=yd_A$ for some $y$.  The process is constructive, and provides a simple algorithm for finding $A$.  Given any $x$, for each $i\in D_R(x)$ we find the largest possible set $D_i:=[i,j]\subsetneq I$ such that:
\[
x=y a_j a_{j-1} \cdots a_i.
\]
We ultimately show (Lemma~\ref{cor:maxRightSet}) that $A$ is the union of the sets $D_i$, and is thus uniquely specified.  But first we need a Lemma describing the relationship between these various sets $D_i$.

\begin{definition}
Let $x$ be an affine permutation, considered as an element of the nilCoxeter monoid, and $i \in D_R(x)$.  Define the set $D_i = [i,j]$ to be the maximal set such that
\[
x= y a_j a_{j-1} \cdots a_i,
\]
for some $y\in {\hat {\mathbb A}}_k$.
\end{definition}

\begin{lemma}
\label{lem:maxRightSet}
Suppose $x$ is an affine permutation, and $A=D_i, B=D_j\subsetneq I$.  Then if $A\cap B\neq \emptyset$, either $A\subset B$ or $B\subset A$.  Furthermore, if $A\cap B=\emptyset$ then $A\cup B$ is not connected.
\end{lemma}
\begin{proof}
We begin by constructing sequences of residues $A_0$ and $B_0$ in the following way.  Set $A_0=(i+n, i+n-1, \ldots, i+1, i)$ with $n$ maximal such that $x=y u_{i+n}\cdots u_i$; we may have $n>k$, so that the $A_0$ contains repetitions of the same index.  Likewise, construct $B_0=(p+m, p+m-1, \ldots, p+1, p)$, such that $m$ is maximal and and $x=y u_{p+m}\cdots u_p$.  Our initial goal will be to show that if $B_0$ and $A_0$ share any indicies, then we must have $A_0\subset B_0$ or vice versa.

As $x$ is an affine permutation, we may consider $x$ as a doubly-infinite sequence of integers without repetitions.  We set $x_l:=x(l)$.  Recall that if $l\in D_R(x)$, we have $x_l>x_{l+1}$, and $x$ has a reduced word ending in $l$.  Since $x=y u_{i+n}\cdots u_{i+1}u_i$, we have $x_i>x_{i+1}$.  Removing the descent at $i$, we obtain $y u_{i+n}\cdots u_{i+1}$, and so this element has a right descent at $i+1$.  But here $x_i$ appears in position $i+1$, and so we have $x_i>x_{i+2}$.  We may continue peeling off generators from the right to show that $x_i>x_k$ for each $k\in \{i+1, i+2, \ldots, i+n+1\}$.

Now we consider two cases.
\begin{itemize}
\item Case 1: $|A_0|, |B_0|\leq k$.  Then we can set $A$ and $B$ to be the set of indices appearing in $A_0$ and $B_0$ respectively.  If $p=D_R(d_B)$ appears in $A$, then $x_i>x_p$.  But $x_p>x_k$ for all $k\in [p+1, p+m+1]$, so that $x_i>x_l$ for all $l\in [i+1, p+m+1]$.  Thus, $B\subset A$.  

\item Case 2: $|A_0|\geq k+1$.  Then $p\in A_0$, so that either $x_i>x_p$ or $x_i>x_{p+k+1}$ if $p<i$ as an integer.  
In the case where $x_i>x_p$, we have $x_{i+k+1}=x_i+k+1>x_p$, so that $i\not \in B_0$.  Additionally, if $i-1$ were in $B_0$, we would have $x_p>x_{i}>x_{i+1}$, which would in turn mean $i\in B_0$.  Thus, $i-1, i\not \in B_0$.  The same reasoning holds if $p>i$ as an integer.  

As such, $B_0$ is a proper subset of the index set, containing neither $i-1$ or $i$.  We set $B=\{p, p+1, \ldots, p+m\}$ and $A=\{i, i+1, \ldots, i-2\}\subsetneq I$.  Then we have shown that $B\subset A$.
\end{itemize}
Thus, if $A\cap B\neq \emptyset$, we have either $B\subset A$ or vice versa.

Now suppose $A\cap B=\emptyset$ and $A\cup B$ connected as a subgraph of the Dynkin diagram, and let $A=[i,j]$ and $B=[p,q]$.  If $|A|, |B|<k$, we have $p=j+1$,and $x_i>x_{j+1}=x_p>x_k$ for all $k\in [p+1,q+1]$.  But then $x_i>x_k$ for all $k\in [i+1, q+1]$, so we can find $C$ such that $A\subsetneq C$ and $x=wd_{C}$, contradicting the maximality of $A$.  If $|A|=k$, we must have $|B|=1$, so we may repeat the same argument and show that $B$ was not maximal.  Thus we have $A\cup B$ disconnected.
\end{proof}

\begin{corollary}
\label{cor:maxRightSet}
For any affine permutation $x$, there exists a unique maximal $A\subsetneq I$ such that $x=y d_A$ with $\len(x)= \len(y) + |A|$.
\end{corollary}
\begin{proof}
Consider $D_R(x)$.  For each $i\in D_R(x)$, we can construct a maximal set $A_i=[i, j]$ for some $j$ such that $x=y_i d_{A_i}$.  
By Lemma~\ref{lem:maxRightSet}, if we consider any pair of these sets, they are either disjoint with their union disconnected, or one is contained in the other.  Thus, the union of the $A_i$ gives a set $A$ such that $x=y d_{A}$ for some $y$.  By construction, $A$ is maximal.

For uniqueness, suppose $B$ is another such set.  Then $D_R(d_B)\subset D_R(x)$; by construction of $A$, we have either $A=B$ or $B\subsetneq A$.  Then maximality of $B$ implies $B=A$.
\end{proof}

\begin{corollary}
For any affine permutation $x\in \nilcox_k$, suppose $A$ is the unique maximal $A\subsetneq I$ such that $x=yd_A$.  Suppose $B\subsetneq I$ and $z\in \nilcox_k$ such that $x=yd_A=zd_B$.  Then $B\subset A$.
\end{corollary}
\begin{proof}
This is a direct consequence of the proof of Lemma~\ref{lem:maxRightSet}.
\end{proof}

\begin{theorem}
\label{thm:maxDecomposition}
Every affine permutation has a unique maximal decomposition into cyclically decreasing elements.
\end{theorem}
\begin{proof}
This follows immediately by repeated application of Corollary~\ref{cor:maxRightSet}.
\end{proof}

\begin{remark}[Algorithm for Computing the Canonical Decomposition.]
\label{rem:algorithm}
The proofs of these results directly translate into an algorithm for finding the canonical decreasing decomposition of any affine permutation $x$.  For each $i\in D_R(x)$ we associate a set $D_i$ obtained by finding the largest connected cyclically decreasing word ending in $u_i$ such that $x = y d_{D_i}$.  Then set $A_1$ to be the union of the sets $D_i$, so that $x = x_1 d_{A_1}$ for some $x_1$.  Repeat this procedure on $x_1$ to obtain $A_2$, and so on.
\end{remark}

\begin{example}
Let $k=7$.  Consider the affine permutation $x$ with base window is $[-4, 1, 2, 0, 5, 14, 7, 11]$.  Then $D_R(x)=\{0, 3, 6\}$.  We form the sets $D_i$: $D_0=\{0,1,2,3\}$, $D_3=\{3\}$, and $D_6=\{6,7,0,1,2,3,4\}$, so that their union $A_1=D_6=\{6,7,0,1,2,3,4\}$.  Then we can find $x_1$ such that $x = x_1 d_{A_1}$.  

This $x_1$ has base window $[1, 2, 0, 5, 6, 7, 11, 4]$.  We have $D_R(x_1)=\{2, 7\}$, and find the sets $D_2=\{2\}$ and $D_7=\{7,0,1,2\}$, so that $A_2 = \{7,0,1,2\}$, and $x_1 = x_2 d_{A_2}$.  

The permutation $x_2$ has base window $[2, 0, 3, 5, 6, 7, 4, 9]$, and $D_R(x_2) = \{6, 1\}$.  Then we form the sets $D_1=\{1\}$ and $D_6=\{6\}$, so that $A_3=\{1, 6\}$.  

Finishing things up, one may derive $A_4=\{5\}$ and $A_5=\{4\}$, so that: 
\[
x = u_{4} u_{5} u_{1, 6} u_{2,1,0,7} u_{4,3,2,1,0,7,6}.
\]  
This is the maximal decomposition of $x$.  This is depicted as a $k$-code filling in Figure~\ref{fig.example1}.

Using a similar algorithm, we may find a cyclically \emph{increasing} decomposition of $x$.  
This decomposition turns out to be $x = u_{5}u_{4}u_{3}u_{2} u_{1,3} u_{0,2} u_{1,7} u_{3,5,6,7,0}$

This is depicted as a $k$-column castle tableau in Figure~\ref{fig.example1}.

\begin{figure}
  \begin{center}
  \includegraphics[scale=0.5]{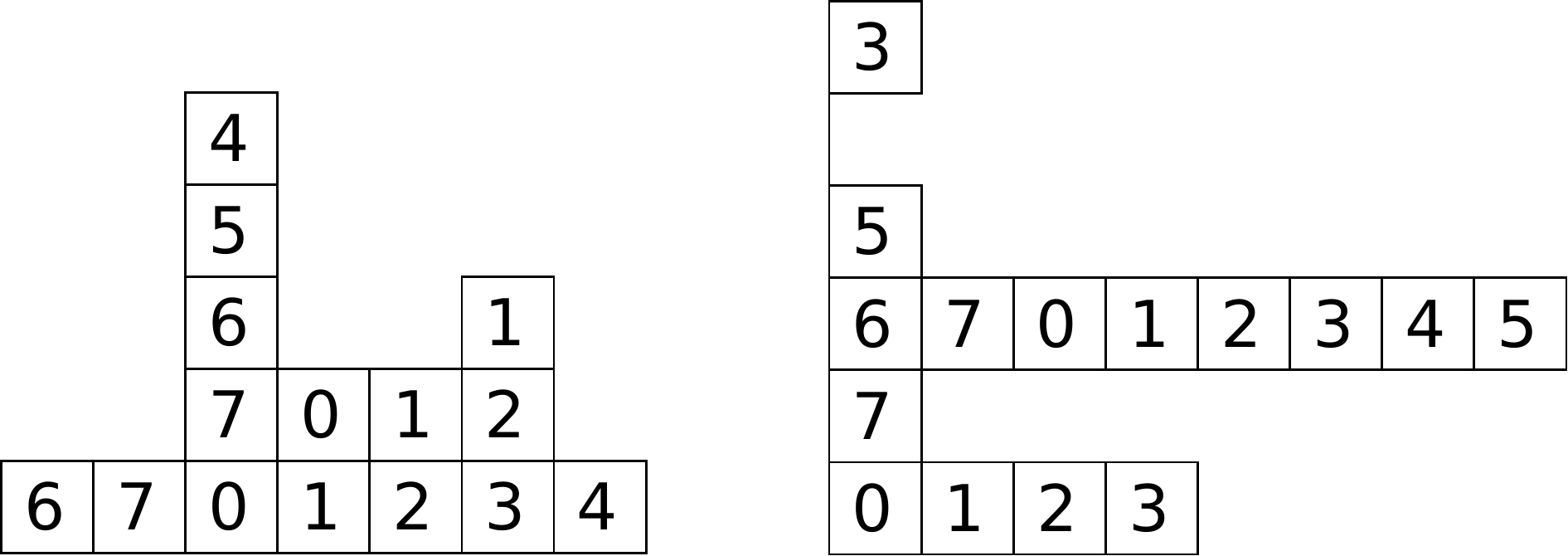}
  \end{center}
  \caption{The maximal cyclically decreasing and increasing $k$-code fillings for the affine permutation with window notation $[-4, 1, 2, 0, 5, 14, 7, 11]$.}
  \label{fig.example1}
\end{figure}
\end{example}

\subsection{Maximizing Moves on $k$-Codes.}
\label{ssec:maxMoves}

Given a non-maximal cyclically decreasing decomposition, there are a number of `moves' we can apply in sequence to obtain the maximal decomposition.  Because of the close link between decompositions and $k$-codes, we will develop these `moves' in tandem in both contexts.  These moves bear some similarity to moves on rc-graphs~\cite{billeyBergeron} or may be thought of as a $k$-bounded variation on jeu de taquin, as they may be used to obtain a $k$-code from a non-maximal $k$-code.

We first examine the action of a single generator applied to a single cyclically decreasing element $d_A$:
\begin{itemize}
\item Commutation.  Suppose $i-1, i, i+1 \not \in A$.  Then $s_{i}d_A=d_{A\cup \{i\}}$.
\item Zero.  Suppose $i\in D_L(d_A)$.  Then $s_{i}d_A=0$.
\item Braid.  Suppose $i\in A, i \not \in D_L(d_A)$.  Then $s_{i}d_A=d_As_{i+1}$.
\end{itemize}
These all follow directly from the definition of the cyclically decreasing elements and the relations in $\nilcox_k$.

Now consider the product of two cyclically decreasing elements, $d_Bd_A$.  Using the above single-generator moves, we establish a number of `moves' for merging elements of $B$ into $A$. This allows us to maximize the vector $(|A|, |B|)$ lexicographically.

\begin{lemma}[Two-Row Moves.]  
\label{lem:twoRowMoves}
The following identities hold for products of cyclically decreasing elements $d_A$ and $d_B$:
\begin{itemize}
\item Commutation.  Suppose $i, i-1, i+1 \not \in A$, and $i\in D_R(d_B)$.  Then $d_Bd_A=d_{B\setminus \{i\}}d_{A\cup \{i\}}$.
\item Chute Move.  Suppose $A=[i,j]$ and $B=[i-1,p]$ with $p\not \in A$ and $j\in B$.  Then $d_Bd_A=d_{B\setminus \{j\}}d_{A\cup \{i-1\}}$.
\item Zero.  Suppose $A=[i,j]$ and $B=[p,q]$ with $p\in A$ and $j\in B$.  Then $d_Bd_A=0$.
\end{itemize}
\end{lemma}
\begin{proof}
The commutation rule follows directly from the single generator moves.  The final two identities follow from applying a sequence of braid and commutation relations in the product.  (And in fact, the Zero move can be derived from the Chute Move.)
\end{proof}

The two-row moves translate directly into operations on (skew) $k$-codes.  In the product, $d_B$ and $d_A$ correspond to two stacked rows containing the residues in $B$ and $A$.  The two-row moves are illustrated in Figure~\ref{fig.2rowMoves}.

Given a skew $k$-code $d_Bd_A$, application of a two-row move reduces the size of $B$ by one and increases the size of $A$ by one.  All of the two-row moves are reversible, and so we also have a set of \emph{reverse two-row moves} which increase the size of $B$ by one and reduce the size of $A$ by one.

\begin{figure}
  \begin{center}
  \includegraphics[scale=0.75]{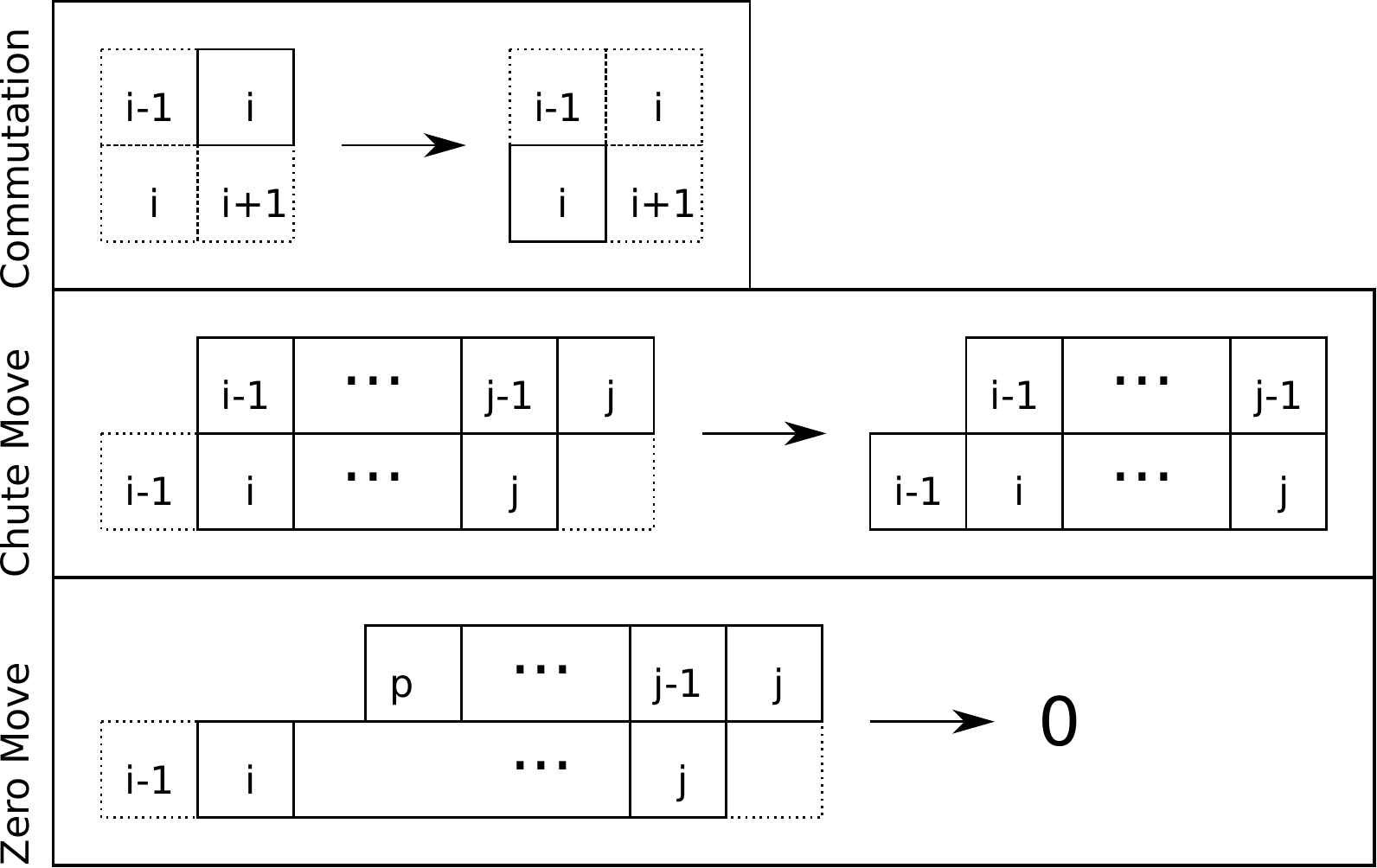}
  \end{center}
  \caption{The two-row moves as they act on $k$-code fillings.}
  \label{fig.2rowMoves}
\end{figure}

We now provide a useful technical lemma with a very nice proof!

\begin{lemma}
\label{lem:2rowEmptyCol}
In any product $d_Bd_A\neq 0$, there exists $j$ such that $j\not \in A$, and $j-1\not \in B$.  In the two-row notation for $d_Bd_A\neq 0$, there is an empty column.
\end{lemma}
\begin{proof}
The two statements are equivalent.  We suppose there is no empty column in the two-row notation, and show that the product $d_Ad_B$ is unreduced.  

Since there is no empty column, we have three possible states for each column.  
\begin{itemize}
\item State $\tp$: $i-1\in B$ and $i-2 \not \in A$,
\item State $\dwn$: $i-1\not \in B$ and $i-2 \in A$, or
\item State $\chrm$: $i-1\in B$ and $i-2 \in A$.
\end{itemize}
Since $B\subsetneq I$, there exists $i\not \in B$; since there is no empty column, this gives $i-1\in A$, so there exists a $\tp$ column.  We now consider each residue $j$ in decreasing order, beginning with $i-1$.  

If the current column is of type $\tp$, one of three cases holds:
\begin{itemize}
\item If $j-1\in B$ and $j-2 \not \in A$, then the product $d_Ad_B$ is unreduced, by the commutation two-row move.
\item If $j-1\in B$ and $j-2 \in A$, the next column if of type $\chrm$.  
\item If $j-1\not \in B$ and $j-2 \in A$, the next column if of type $\tp$.  
\end{itemize}
So the next column is either of type $\chrm$ or $\tp$.

If the current column is of type $\chrm$, one of three cases holds:
\begin{itemize}
\item If $j-1\in B$ and $j-2 \not \in A$, then the product $d_Ad_B$ is unreduced, by the chute move.
\item If $j-1\in B$ and $j-2 \in A$, the next column if of type $\chrm$.  
\item If $j-1\not \in B$ and $j-2 \in A$, the next column if of type $\tp$.  
\end{itemize}
So the next column is either of type $\chrm$ or $\tp$.

So for every column, the next column is of type $\chrm$ or $\tp$.  Both of these cases have the residue $j-1 \in A$, so that every residue must be in $A$. But $A\subsetneq I$, providing a contradiction.  
\end{proof}

For any $S=\{s_1, \ldots, s_i\} \subset I$, let $S-1$ denote the set $\{s_1-1, s_2-1, \ldots, s_i-1\}$.
\begin{lemma}
\label{lem:2rowReduction}
Given two sets $A,B\subsetneq I$ with $d_Bd_A\neq 0$, there exist sets $A', B'$ such that $d_Bd_A=d_{B'}d_{A'}$ and $B'\subset A'-1$.  In particular, in the $k$-code filling for $d_{B'}d_{A'}$, every residue in $B'$ sits directly above a residue in $A'$.
\end{lemma}
\begin{proof}
We establish an explicit algorithm for maximizing the product $d_Ad_B\neq 0$ using a sequence of two-row moves.

By Lemma~\ref{lem:2rowEmptyCol}, there exists a residue $i$ such that $i\not \in A, i-1\not \in B$.  
We set $E:=i$ to be the current empty column.  From the current empty column, we will read columns in increasing order.  If the next column is empty, we set $E:=E+1$ to be the current empty column and continue.  Otherwise, we have one of three possibilities for the adjacent column:
\begin{itemize}
\item $\tp$: We apply a commutation move.  The current empty column becomes of type $\dwn$, and the next column becomes empty.  We set the current empty column to the next column.

\item $\dwn$: We set $N=E+1$, and continue reading to the right incrementing $N$ to keep track of the current non-empty column.  If column $N+1$ is of type $\tp$, the product is unreduced.  If column $N+1$ is empty, we set $E:=N+1$ and continue.  If column $N+1$ is of type $\dwn$, we set $N:=N+1$ and continue.  

The last case is when $N+1$ is of type $\chrm$; in this case we keep reading (set $N:=N+1$), but have a new set of possibilities.  The next column $N+1$ may be of type $\chrm$ or $\dwn$, either of which is ok: set $N:=N+1$ and carry on.  If the next column is of type $\tp$, then the product is unreduced by the chute move.  Finally, if the next column is empty, we set that column to be the current empty column $E:=N+1$ and continue.

\item $\chrm$: We set $N:=E+1$ and read columns as in the case $\dwn$.  The only difference is that if we meet a $\tp$ column before meeting a $\dwn$ column, we may apply a chute move.  Then the current empty column becomes of type $\dwn$, and the $\tp$ column becomes empty.  We set the current empty column to be the newly created empty column, and continue.
\end{itemize}

In all cases where a box is moved, a box moves from the top row to the bottom row.  This implies that this process must stabilize at some point.  In all cases we eliminate columns of type $\tp$, so that the final expression will contain no $\tp$ columns.  Thus, $A'\subset B'-1$.
\end{proof}

\begin{example}
Let $k=9$, with $A=\{2,3, 5,8\} and B=\{0,1,2,3,4,7,8\}.$  We find $A', B'$ such that the product $d_{B'}d_{A'}$ is maximal.  We apply a series of moves:
\begin{eqnarray*}
d_B d_A &=& (a_{87} a_{43210}) (a_{8}a_5 a_{32}) \\
        &=& (a_{87} a_{4321}) (a_{8}a_5 a_{32}a_0) \text{(commutation)} \\
        &=& (a_{87} a_{4} a_{21}) (a_{8}a_{5} a_{3210}) \text{(chute move)} \\
        &=& (a_{7} a_{4} a_{21}) (a_{87}a_{5} a_{3210}) \text{(chute move)}. \\
\end{eqnarray*}
Thus, we have $A'=\{0,1,2,3,5,7,8\}$ and $B'=\{1,2,4,7\}$.
\end{example}

For a product of more than two cyclically decreasing elements $d_{\vec{A}}$, we may progressively apply two-row moves to pairs of adjacent cyclically decreasing elements, eventually obtaining a decomposition with $A_{i+1}\subset A_i-1$ for each $i$.  Such a decomposition can be represented by a $k$-code fillings by selecting in row $i$ the residues in $A_i$.  Thus, we have the following Lemma:

\begin{proposition}[Maximal Cyclic Products]
\label{prop:maxCyclic}
For any $w\in \nilcox_k$, if $w=d_{\vec{A}}$ is a maximal expression for $w$ as a cyclically decreasing product then $\vec{A}$ has $A_{i+1}\subset A_i-1$ for each $i$.  In particular, we observe that $\sh(\vec{A})$ is a partition.
\end{proposition}

Thus, we have shown that any reduced decomposition may have a series of two-row moves applied to it to obtain a decomposition corresponding to a $k$-bounded tableau.

\subsection{Insertion Algorithm.}
\label{ssec:insertion}

Consider $x$ an affine permutation with $x=d_{\vec{A}}$, $\vec{A}=\{A_1, \ldots, A_n\}$, giving the maximal decomposition of $x$.  We consider the product $xa_p$ for $p\not \in D_R(x)$, and find an algorithm for determining the $k$-code filling $T'$ for $xa_p$.  To do this, we attempt to insert the residue $p$ into the set $A_j$, beginning with $j=1$.  One of following possibilities occurs:
\begin{itemize}
\item (Inclusion I.)   If $p-1, p, p+1 \not \in A_j$: By the commutation relation, we may include $p$ into $A_j$.  Include $p$ into $A_j$ and halt the algorithm.
\item (Inclusion II.)   If $p-1, p \not \in A_j$, but $p+1\in A_j$:  We have $p+1\in D_R(A_j)$, and may include $p$ into $A_j$.  So again, include $p$ into $A_j$ and halt the algorithm.
\item (Bump Move.)   If $p \not \in A_j$ and $p-1 \in A_j$, we have $d_{A_j}u_p = u_{p-1}d_{(A_j\setminus \{p-1\})\cup \{p\}}$.  In other words, bump the residue $p-1$ from $A_j$ and replace it with the residue $p$.
\item (Braid Move.)  If $p \in A_j$ but $p \not \in D_R(A_j)$, then $d_{A_j}u_p=u_{p-1}d_{A_j}$ by a braid relation.  In this case, leave $A_j$ unchanged and continue the process, trying to insert $p-1$ into $A_{j+1}$.
\end{itemize}
These cover all possibilities.  When the product is non-zero, this gives us a way to insert a new box into the $k$-code.

We remove the explanations from the different cases to obtain a reduced list of insertion moves:
\begin{itemize}
\item (Inclusion.)   If $p-1, p \not \in A_j$:  Include $p$ into $A_j$ and halt the algorithm.
\item (Bump Move.)   If $p \not \in A_j$ and $p-1 \in A_j$, Remove the residue $p-1$ and include residue $p$ in $A_j$.  Continue the insertion with the residue $p-1$ into row $j+1$.
\item (Braid Move.)  If $p \in A_j$ but $p \not \in D_R(A_j)$:  Leave row $A_j$ unchanged, and continue the insertion algorithm with residue $p-1$ into row $j+1$.
\end{itemize}

\begin{definition}
Let $S$ be a $k$-code and $p$ a residue.  We denote the insertion of $p$ into $S$ by $S \leftarrow p$.  
\end{definition}

Notice that in both the braid move and the bump move, the residue $p-1$ is in the (possibly modified) $A_j-1$.  As a result, inserting a residue $i$ into a $k$-code will produce another $k$-code, so long as the product $xu_i\neq 0$.  Luckily, we can use Corollary~\ref{cor:descents} to read off the right descents of $x$ from its $k$-code, making it immediately clear whether a given value can be inserted or not.

We may form a \emph{recording tableau} $Q$ in the usual way.  Suppose $w=[w_1, \ldots, w_n]$ is a word in the alphabet $I$ which inserts to a $k$-code $P$.  On inserting the $j$th letter of $w$, we write a $j$ in the final box in the insertion of $[w_1, \ldots, w_j]$.  The only special case is the bump move, which replaces the box with residue $p$ with the box with residue $p+1$.  Suppose the residue $p$ box was marked with an $l$ in the recording tableau: We simply put this $l$ in the box with residue $p+1$ and delete the box with residue $p$.  (This is illustrated in Figure~\ref{fig.affineInsert})

Any reduced word $w$ for a given permutation $x$ may be inserted to the empty $k$-code to obtain a tableau $Q$ which depends on the reduced word that was inserted.  In fact, all of the insertion moves are invertible, allowing a reverse insertion algorithm.  Then given a recording tableau $Q$ one may recover the reduced expression $w$. 

\begin{theorem}
\label{thm:redWords}
Let $\mathfrak{Q}=\{Q_i\}$ be the set of recording tableaux associated to a $k$-code $\alpha$ obtained from a maximal decomposition of an affine permutation $x$. Then $\mathfrak{Q}$ is in bijection with the set of reduced words for $x$.
\end{theorem}

Call a recording tableau $Q$ \emph{standard} if it arises as the recording tableau of some reduced expression for an affine permutation.  Then it is clear that there is a bijection between standard recording tableaux of a given shape and reduced expressions for the affine permutation with the associated $k$-code.

\begin{problem}
Find a combinatorial description of the recording tableaux.
\end{problem}

\begin{example}
At $k=1$, there are no relations between the generators $u_0$ and $u_1$.  In this case, there are exactly two $k$-codes of size $n$ ($(n,0)$ and $(0,n)$), and exactly two non-zero words on $n$ letters, one with right descent $0$ and one with right descent $1$.
\end{example}

\begin{example}
With $k=3$, let $w=[0,3,1,2,1,0]$.  Then the insertion of $w$ is the $k$-code $\alpha=(2,1,3,0)$.  But the recording tableau has first row $[2,5,1]$, which is not standard in the usual sense for tableaux.
\end{example}

\begin{figure}
  \begin{center}
  \includegraphics[scale=0.5]{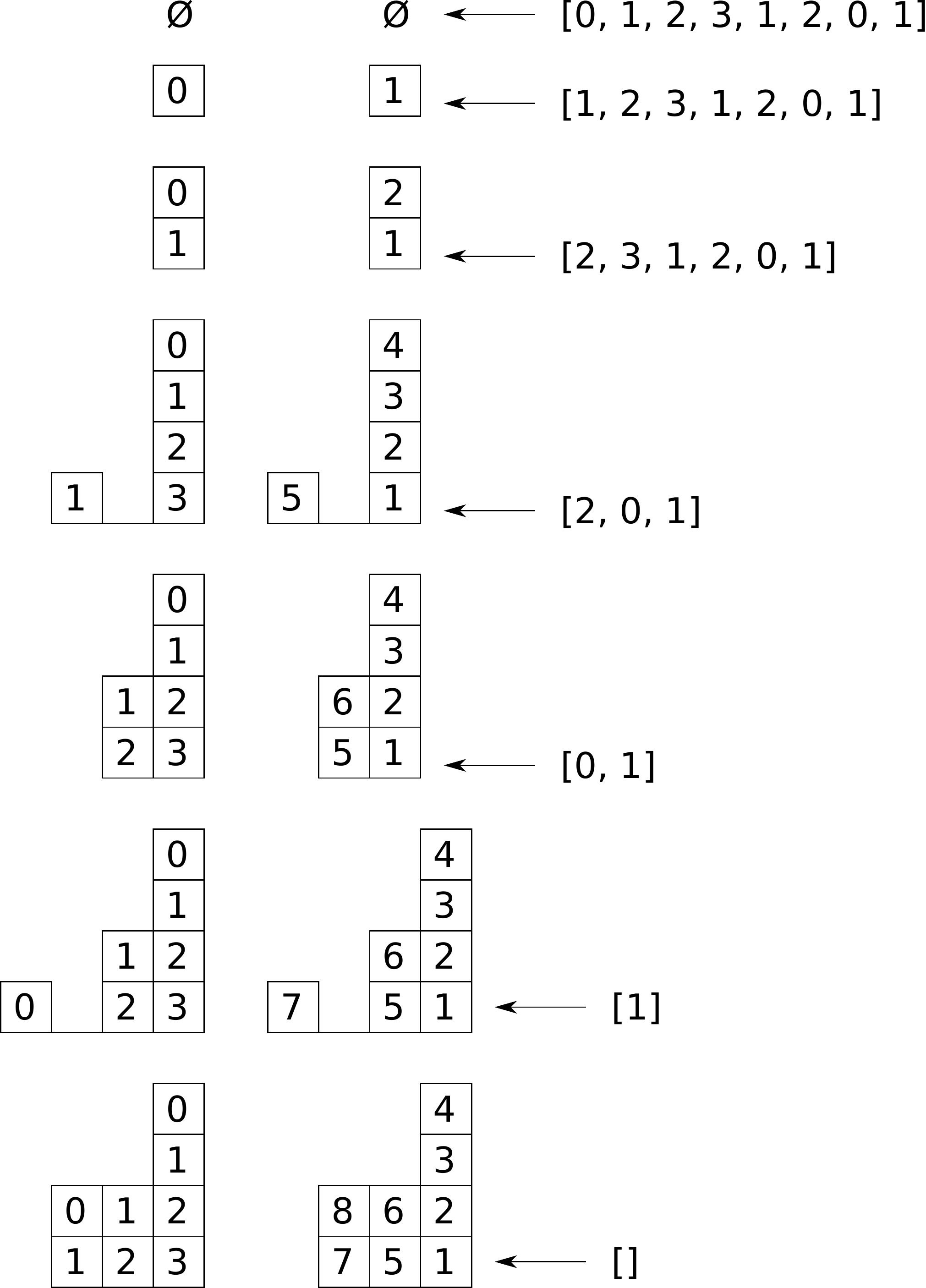}
  \end{center}
  \caption{The insertion algorithm for $k=4$ with the word $w=[0, 1, 2, 3, 1, 2, 0, 1]$.  The left tableaux are the $k$-code fillings obtained at each step, and the right tableaux are the recording tableaux.}
  \label{fig.affineInsert}
\end{figure}

\subsection{Bijection Between Affine Permutations and $k$-Codes.}
\label{ssec:bijection}

Our goal in this section is to prove the bijection between $k$-codes and affine permutations.  We begin by restating the results of Theorem~\ref{thm:maxDecomposition} and Proposition~\ref{prop:maxCyclic} in a consolidated statement:
\begin{theorem}[Canonical Cyclically Decreasing Decomposition.]
\label{thm:canoncialDecomposition}
Every affine permutation $x$ admits a unique maximal decomposition as a product of cyclically decreasing elements $x=d_{\vec{A}}$.  This decomposition has $A_{i+1}\subset A_i-1$ for each $i$, and thus $\sh(\vec{A})$ is a partition.
\end{theorem}
\begin{proof}
This follows from repeated application of Corollary~\ref{cor:maxRightSet} to obtain a complete decomposition of the affine permutation $x$ as a product of maximal cyclically decreasing elements.  By construction, this decomposition is maximal.  It must also satisfy $A_{i+1}\subset A_i-1$ for each $i$, or else we could apply a two-row move to obtain a new decomposition greater in lexicographic order.  
\end{proof}

\begin{definition}
We refer to the maximal decomposition of $x$ as the \emph{canonical decreasing decomposition} of $x$, denoted $\RD(x)$.  The corresponding maximal decomposition into cyclically increasing elements is the \emph{canonical increasing decomposision} of $x$, denoted $\RI(x)$.

We define a map $\sigma$ from affine permutations to $k$-code fillings.  For $x$ an affine permutation with canonical decreasing decomposition $x=d_{\vec{A}}$, take $\sigma(x)$ to be the $k$-code filling whose $i$th row is given by the set of residues $A_i$.
\end{definition}

\begin{definition}
A \emph{descent} of a $k$-code $\alpha$ is an index $i$ such that $\alpha_{i-1}<\alpha_{i}$.
\end{definition}

\begin{corollary}[Descent Sets from $k$-code fillings.]
\label{cor:descents}
Given a maximal $k$-code filling $T=\sigma(x)$ for an affine permutation $x=d_{\vec{A}}$, then $r\in D_R(x)$ if and only if $r$ appears in the first row of $T$ and the column containing this box contains a right descent for one of the $d_{A_i}$.
\end{corollary}
\begin{proof}
These descents occur by repeated use of the braid relation to move a right descent in $d_{A_i}$ to the beginning of a reduced expression for $x$.
\end{proof}

\begin{corollary}
\label{cor:containment}
Let $x$ be an affine permutation with decomposition $x=d_{\vec{A}}$.  Then this decomposition is maximal if and only if $A_{i+1} \subset A_i-1$ for every $i$.
\end{corollary}
\begin{proof}
The forward direction is given by Lemma~\ref{prop:maxCyclic}.

On the other hand, if $A_{i+1} \subset A_i$ for all $i$, we may apply the algorithm in Remark~\ref{rem:algorithm} to obtain a maximal decomposition $x=d_{\vec{B}}$.  We can also associate a $k$-code filling $T$ to the decomposition $d_{\vec{A}}$.  The algorithm constructs sets $D_j$ for each $j\in D_R(x)$ and takes $B_1$ to be the union of the sets $D_j$.  By Corollary~\ref{cor:descents}, we may then observe that $A_1=B_1$.  We may then repeat this process to show that $A_i=B_i$ for every $i$.  Thus, $d_{\vec{A}}$ is the maximal decomposition of $x$.
\end{proof}

\begin{theorem}
\label{thm:compTabBijection}
The set of $k$-codes is in bijection with affine permutations in $\AS_{k+1}$.
\end{theorem}
\begin{proof}
The map $\sigma$ takes permutations of length $n$ to $k$-codes with $n$ boxes, so we may consider $\sigma$ as a graded map on finite sets.  Additionally, we may also recover $x$ by taking the reading word of $\sigma(x)$.  By Corollary~\ref{cor:containment}, for any $x\neq y$, we have $\sigma(x)\neq \sigma(y)$, so $\sigma$ is one-to-one.  Then we only need to show that every $k$-code $T$ is a maximal decomposition of some affine permutation; equivalently, that element obtained by the reading word of $T$ is non-zero in $\nilcox_k$.  

For this, we induct on the number of boxes $n$ in  $T$.  At $n=1$, the single box corresponds to a simple transposition, and the statement holds.  Suppose that $S$ is the tableau of shape $\alpha$ with $n+1$ boxes and $\alpha_{j-1}=0, \alpha_{j}\neq 0$.  Then $j$ is a descent of $S$.  Removing the box $j$, we may apply a sequence of commutation two-row moves to remove one box from column $j$ and shift it to position $j-1$.  Since $S$ was a $k$-code, the resulting $S'$ is also a $k$-code, but on $n$ boxes.  As such, it is equal to $\sigma(x)$ for some $x$.  This $x$ has $j\not \in D_R(x)$, so $xu_j\neq 0$.  Reinserting $j$ into $S'$ yields $S$, so we see that $\sigma(xu_j)=S$.  This completes the proof.
\end{proof}

\begin{corollary}
Consider an affine permutation $x$ with $k$-code filling $T$.  Then $x$ is $i$-dominant if and only if $T$ has a flattening which is a $k$-bounded partition with residue $i$ in the lower left box.
\end{corollary}
\begin{proof}
This follows immediately from Corollary~\ref{cor:descents}.
\end{proof}

Note that this gives an alternate proof of the bijection between $k$-bounded partitions and $0$-dominant (or Grasssmannian) elements.  

This allows us to prove a theorem on the nil-Coxeter realization of the $k$-bounded symmetric functions.

We complete this subsection with a simple statement relating the $k$-codes to symmetric functions in non-commuting variables.
\begin{theorem}
\label{thm:maxhCoefficient}
Suppose that $\sh(\sigma(x))=\lambda$.  Then:
\[
[x]h_{\lambda} = 1, \text{ and } [x]s_{\lambda}^{(k)} = 1.
\]
\end{theorem}
\begin{proof}
In both $h_{\lambda}$ and $[x]s_{\lambda}^{(k)}$, all coefficients are integers $\geq 0$.  We have $[x]h_{\lambda}=1$ by the uniqueness of the cyclic decomposition of $x$.  Furthermore, $h_\lambda = s_{\lambda}^{(k)} + \sum_\mu s_\mu^{(k)}$, where each $\mu$ dominates $\lambda$.  If we had some $\mu$ with $[x]s_{\mu}^{(k)}=1$, we could then have $[x]h_{\mu}=1$, contradicting the maximality of the decomposition of $x$.
\end{proof}

\subsection{Relating the Various Cyclic Decompositions of an Affine Permutation.}  
\label{ssec:relations}

The constructions of this section may be modified to provide four different $k$-codes associated to any affine permutation $x$.  These are obtained by finding maximal cyclically increasing and decreasing decompositions for $x$ from the right and from the left.  The decomposition from the left finds $x=d_{A_n}d_{A_{n-1}}\cdots d_{A_1}$ maximizing $(|A_n|, |A_{n-1}|, \ldots, |A_1|)$ lexicographically.  This may be found by modifying the algorithm for generating a $k$-code to consider the left descents of $x$ instead of the right descents.

\begin{definition}
Let $x$ be an affine permutation.  Set $\RD(x)$ to be the $k$-code corresponding to the maximal right decomposition of $x$ into a product of cyclically decreasing elements.  Likewise, set $\RI(x)$ to be the $k$-code from the right increasing decomposition, and $\LD(x)$ and $\LI(x)$ be the $k$-codes from the left decreasing and increasing decompositions, respectively.
\end{definition}

\begin{example}
\label{ex:kcastlecomputation}
Let $k=3$, and $x=a_{2, 1, 0, 3, 0, 1, 2, 1, 0, 3, 1, 2, 0, 1, 0}$.
Then $x$ has the following maximal cyclic decompositions:
\begin{itemize}
\item Decreasing Right:  \[x=a_2a_3a_0a_1a_{3,2}a_{0,3,2}a_{1,0,3}a_{2,1,0},\] so   $\RD(x)=(3,8,4,0)$.
\item Increasing Right:  \[x=a_2a_1a_0a_3a_2a_1a_0a_3a_{2,3}a_{1,2}a_{3,0,1},\] so  $\RI(x)=(11,3,0,1)$.
\item Decreasing Left:  \[x=a_{2,1,0}a_{3,2,1}a_{0,3,2}a_{3,1}a_2a_3a_0a_1,\] so $\LD(x)=(4,3,8,0)$.
\item Increasing Left:  \[x=a_{2,3,1}a_{1,3}a_{0,2}a_3a_2a_1a_0a_3a_2a_1a_0,\] so   $\LI(x)=(3,0,11,1)$.
\end{itemize}
\end{example}

An alternative way to produce $\LD(x)$ from $\RD(x)$ is to use the reverse two-row moves to `up-justify' $\RD(x)$.  The resulting object's reading word gives the left decomposition of $x$ into cyclically decreasing elements.

We can establish a more direct relationship between $\LD(x)$ and $\RD(x)$.  
\begin{proposition}
\label{prop:colPermutation}
$\RD(x)$ is a permutation of $\LD(x)$, and $\RI(x)$ is a permutation of $\LI(x)$.
\end{proposition}
\begin{proof}
Suppose $x=d_{\vec{A}}$ with $\sh(\vec{A})=\lambda$.  Recall that $h_\lambda=h_{\lambda_n}\cdots h_{\lambda_1}$.  By Theorem~\ref{thm:maxhCoefficient}, we have $[x]h_{\lambda} = 1$, corresponding to the unique maximal cyclically decreasing decomposition of $x$.  But because the $h_i$'s commute, we have $[x]h_{\lambda_1}\cdots h_{\lambda_n} = 1$.  Since $x$ appears in $h_{\lambda_n}\cdots h_{\lambda_1}$, there exists a cyclically decreasing decomposition for $x=d_{\vec{B}}$ of shape $(\lambda_n, \ldots, \lambda_1)$.  This decomposition is maximal as a left cyclically decreasing decomposition, or else commutativity of the $h_i$'s would imply that our original decomposition of $x$ was not maximal.  Then $\sh(\vec{B})$ is the reverse of $\lambda$, and $B_{i-1}\subset B_i-1$ for each $i$, implying that the entries in $\LD(x)$ are the same as the entries in $\RD(x)$, up to some reordering.

The proof for the increasing case is identical.
\end{proof}

\begin{problem}
Describe the permutation relating $\RD(x)$ and $\LD(x)$ for arbitrary $x$.  Is there a straightforward way to calculate the permutation, short of directly computing the left maximal decomposition?
\end{problem}

\begin{lemma}
For any affine permutation $x$, the descent sets of $\RI(x)$ and $\RD(x)$ are equal.  Also, the descent sets of $\LI(x)$ and $\LD(x)$ are equal.
\end{lemma}
\begin{proof}
The descent sets of $\RI(x)$ and $\RD(x)$ are equal to $D_R(x)$, by Corollary~\ref{cor:descents}.  One can prove an analog of Corollary~\ref{cor:descents} for the left decompositions, giving the second statement.
\end{proof}

Recall that the inverse of an affine permutation $x$ is obtained by reversing any reduced word for $x$.

\begin{proposition}
Let $x$ be an affine permutation.  Then $\RD(x^{-1})=\LI(x)$ and $\LD(x^{-1})=\RI(x)$.  
\end{proposition}
\begin{proof}
The element $x^{-1}$ is obtained by reversing a reduced expression for $x$.  The reversal of a cyclically decreasing element is a cyclically increasing element, and vice versa.  Thus, reversing the cyclic decomposition immediately converts the maximal decreasing right decomposition for $x$ into the maximal increasing left decomposition for $x$ (which coincides with the maximal right increasing decomposition of $x^{-1}$).
\end{proof}


\subsection{Affine Codes and $k$-Codes}  
\label{ssec:affLehmer}

The various $k$-codes associated to an affine permutation relate directly to the affine code derived from considering $x$ as a permutation of the integers.  There are various ways to construct the code of a permutation in the finite case; we directly generalize four methods and place them in correspondence with the $k$-codes.  Two of these methods correspond to the affine code of the permutation, and two correspond to the inversion vector.  We unify the two concepts by referring to both as simply the code of the permutation.

\begin{definition}
An \emph{affine code} is given by a vector with $k+1$ entries, $L=\{L_1, \ldots, L_{k+1}\}$.  Four different affine codes are described below, by providing an algorithm for finding the $i$th entry of the code for an affine permutation $x$.
\begin{itemize}
\item CRD: The \emph{right decreasing code} is given by the number of $j<i$ with $x(j)>x(i)$.
\item CRI: The \emph{right increasing code} is given by the number of $j>(i+1)$ with $x(j)<x(i+1)$.
\item CLD: The \emph{left decreasing code} is given by the number of $j<x^{-1}(i+1)$ with $x(j)>i+1$.
\item CLI: The \emph{left increasing code} is given by the number of $j>x^{-1}(i)$ with $x(j)<i$.
\end{itemize}
\end{definition}

\begin{example}
Consider the affine permutation $x$ with $k=3$ and window notation $[1, -6, 0, 15]$.  
Then the CRD is $(3,8,4,0)$: for example, $11$, $7$ and $3$ all appear to the left of $1=x(1)$, so that the first entry of $\LCRD(x)$ is $3$.  This matches the $k$-code for this element, described in Example~\ref{ex:kcastlecomputation}.
\end{example}

\begin{proposition}
The affine codes described above are equal to the respective $k$-codes for an affine permutation.
\end{proposition}
\begin{proof}
We show that $\LCRD(x)=\RD(x)$; the other three equalities follow similar logic.  Let $L:=\LCRD(x)=(L_1, \ldots, L_{k+1})$, and let $K:=\RD(x)=(K_1, \ldots, K_{k+1})$.  We show that $L=K$.

When $x=d_A$ for some $A$, the statement is clear:  For any $i\in I$, if $i\not \in A$ then there are no larger elements to the left of position $i$, so $L_i=0$, as will $K_i$.  If $i\in A$, the transposition at positions $i, i+1$ moves exactly one large element to the left of position $i$, so that $L_i=1$.  Likewise, $K_1=1$ in this case, so the base case holds.

For the induction step, let $x=yd_A$ be a reduced product, with $A$ maximal.  By induction, we have $\LCRD(y)=\RD(y)$.  By inspection, applying $d_A$ to $y$ either increases by one or stabilizes each entry $L_i$ in $\LCRD(y)$, according to whether $i\in A$.  Then the proposition holds.
\end{proof}

\subsection{Grassmannian and $i$-Dominant Elements.}  
\label{ssec:grassmannian}

A special case, important in the study of $k$-Schur functions, occurs when an affine permutation $x$ has $D_R(x)\subset \{i\}$ for some $i\in I$.  When $i=0$, $x$ is called a \emph{Grassmannian element}, and otherwise it is known as an $i$-dominant element.  By Corollary~\ref{cor:descents} these are given by the $k$-codes with (at most) a single decent at position $i$; a flattening of such a $k$-code is a $k$-bounded partition.  

The following result is known within the community (in particular to the authors of~\cite{bbtz2011}), but the author has been unable to find a reference.  We state the result here as a corollary of the $k$-code construction.
\begin{corollary}
\label{cor:cycDecrWord}
Let $w_\lambda$ be the $0$-dominant element in the expansion of $s_\lambda^{(k)}$.  Then $w_\lambda$ has a unique reduced decomposition as a maximal cyclically decreasing product, where the $i$-th cyclically decreasing element has length $\lambda_i$.  This word is obtained by writing the diagram of the $k$-bounded partition of $\lambda$ and marking the $k+1$-residues in each box, and then reading the rows of the resulting tableau right-to-left.  In other words, if $\lambda$ has $n$ parts,
\[
w_\lambda = \prod_{i=1}^{n} d_{ [-n+i, -n+i+\lambda_{n-i}-1 ]},
\]
where the subscripts are considered modulo $k+1$.
\end{corollary}

An identical argument allows one to find a reduced word for $w_\lambda$ as a maximal cyclically increasing product.  To find this reading word, consider the bijection between $k$-bounded partitions and $k+1$-cores.  The $k$-bounded partition is obtained by removing all boxes with hook $>k+1$ from the $k+1$-core, and then ``left-justifying'' the resulting skew shape (called the $k$-boundary of the core).  To obtain a cyclically increasing word for $w_\lambda$, one instead ``down-justfies'' the $k$-boundary to obtain a partition whose columns are all $k$-bounded.  Fill the boxes of this partition with $k+1$ residues, and then read the columns top-to-bottom, right-to-left.  

Let $\lambda$ be a $k$-bounded partition.  Then the bijection between $k$-bounded partitions and $k+1$ cores yields a core $\mu$.  The bijection between $k+1$-cores and $k$-column bounded partitions gives us a $k$-column bounded partition $\nu$.  To all of these things, there is a $0$-dominant element $w\in \nilcox_k$.  We can read off the maximal cyclically decreasing product for $w$ from $\lambda$, and the maximal cyclically increasing product for $w$ from $\nu$.

In particular, we can convert very quickly between maximal cyclically increasing and decreasing expressions for $w$.

If we wish to find an $i$-dominant maximal cyclically decreasing (resp. increasing) word, we can simply add $i$ to all the residues in $\lambda$ (resp. $\nu$); this is equivalent to applying the Dynkin diagram automorphism $i$ times to the word for the $0$-dominant element $w$.

Suppose $x$ is an $i$-dominant affine permutation with $k$-code $\alpha=\RD(x)$ which flattens to the $k$-bounded partition $\lambda$.  Then $x$ also has a $k$-column castle $\beta=\RI(x)$, which also has (at most) one descent, and is thus also associated to a $k$-bounded partition $\mu$.  These two partitions are related by an operation called the $k$-conjugate.  There is a bijection $\mathfrak{c}$ from $k$-bounded partitions to $k+1$ cores, which are partitions containing no hooks of length $k+1$.  Denote the core associated to a partition $\nu$ by $\mathfrak{c}(\nu)$, and the conjugate of a partition by $\nu^t$.  Then the $k$-conjugate of $\lambda$ is defined to be $\lambda^{(k)}:=\mathfrak{c}^{-1}((\mathfrak{c}(\lambda))^t)=\mu$.  

We summarize this discussion in the following proposition:
\begin{proposition}
\label{prop:kconjagrees}
Let $x$ be a $0$-dominant affine permutation, associated to $k$-bounded partition $\lambda$ with column heights given by $\lambda'=(\lambda_1', \ldots, \lambda_k')$, some of which may be zero.  Then the $k$-code $\RD(x)=(\lambda_1', \ldots, \lambda_k',0)$.  Furthermore, if $\nu$ is the $k$-conjugate of $\lambda$, with column heights $(\nu_1', \ldots, \nu_k')$, we have $\RI(x)=(\nu_1', \ldots, \nu_k',0)$.
\end{proposition}

\begin{example}
Consider the $3$-bounded partition $\lambda=(3,2,2,1,1)$.  Below we see $\lambda$, the associated $(3+1)$-core, and $3$-column bounded partition:
\vspace{.1in}

\begin{center}
\begin{tikzpicture}[scale=.5]
\coordinate (base point) at (0,0);
\node (tab0) at (1,1) [rectangle, minimum size=1cm] {};
\foreach \entry / \r / \c in {  0/0/0, 1/0/1, 2/0/2, 
                                3/1/0, 0/1/1, 
                                2/2/0, 3/2/1, 
                                1/3/0, 
                                0/4/0 }
{ \draw (base point) ++(\c,\r) rectangle +(1,1) node [midway] {\entry}; }


\coordinate (base point) at (6,0);
\node (tab0) at (1,1) [rectangle, minimum size=1cm] {};
\foreach \entry / \r / \c in { x/0/0, x/0/1, x/0/2, +/0/3, +/0/4, +/0/5,
                               x/1/0, +/1/1, +/1/2,  
                               x/2/0, +/2/1, +/2/2, 
                               +/3/0, +/4/0}
{ \draw (base point) ++(\c,\r) rectangle +(1,1) node [midway] {\entry}; }

\coordinate (base point) at (15,0);
\node (tab0) at (1,1) [rectangle, minimum size=1cm] {};
\foreach \entry / \r / \c in {  0/0/0, 1/0/1, 2/0/2, 3/0/3, 0/0/4, 1/0/5, 
                                3/1/0, 0/1/1, 1/1/2 }
{ \draw (base point) ++(\c,\r) rectangle +(1,1) node [midway] {\entry}; }
\end{tikzpicture}
\end{center}
\vspace{.1in}

Then the maximal cyclically decreasing decomposition for $w_\lambda$ is $a_{0}a_{1}a_{3,2}a_{3,0}a_{2,1,0}$, obtained by reading the residues in the rows from right to left, top to bottom.  The maximal cyclically increasing decomposition is $a_{1}a_{0}a_{3}a_{1,2}a_{0,1}a_{3,0}$.
\end{example}

The constructions of this section provide a natural generalization of the $k$-conjugate to arbitrary affine permutations.

\begin{definition}
\label{def:kconj}
If $x$ has $\RD(x)=\alpha$ and $\RI(x)=\beta$, then we say that $\alpha$ and $\beta$ are \emph{$k$-conjugates}, and write $\alpha^{(k)}=\beta$.  Additionally, we define the $k$-conjugate of $x$ to be the affine permutation $x^{(k)}$ with $\RD(x^{(k)})=\beta$ and $\RI(x^{(k)})=\alpha$.
\end{definition}

The following proposition is then immediate.

\begin{proposition}
The $k$-conjugate induces an involution on the affine symmetric group.  This involution preserves length and right descent sets of affine permutations.
\end{proposition}

\subsection{Generalized Pieri Rule.}
\label{ssec:genpieri}

There is a combinatorial Pieri Rule on $k$-bounded partitions which corresponds to the Pieri Rule for $k$-Schur functions.

In this subsection, we generalize the combinatorial Pieri Rule to $k$-codes, and general affine permutations.  First, we establish the notion of skew $k$-codes.

\begin{definition}
Let $\alpha$ and $\beta$ be $k$-codes.  We say that $\beta$ \emph{contains} $\alpha$, $\alpha \subset \beta$, if $\alpha_i\leq \beta_i$ for every $i \in I$.  We define a \emph{skew $k$-code} to be a pair $(\beta, \alpha)$ where $\alpha \subset \beta$.  The \emph{tableau} of a skew $k$-code is the $k$-code filling of $\beta$ with all boxes from $\alpha$ removed.

We say that a skew-$k$-code is a \emph{horizontal strip} if the tableau of $(\beta, \alpha)$ contains no more than one box in each column.  Likewise, $(\beta, \alpha)$ is a \emph{vertical strip} if its tableau contains no more than one box in each row.
\end{definition}

\begin{proposition}
\label{prop:containment}
Let $x$ and $y$ be affine permutations with $xy\neq 0$.  Then $\RD(y)\subset\RD(xy)$.
\end{proposition}
\begin{proof}
Given the $k$-codes $\RD(x), \RD(y)$ we may obtain $\RD(xy)$ by stacking the two castle tableaux appropriately and applying a sequence of two-row moves to obtain a maximal decomposition of $xy$.  In the  application of two-row moves, the lower row is always preserved as columns of type $\tp$ are eliminated.  Since $\RD(y)$ contains no pairs of adjacent rows of type $\tp$, we then observe that the $k$-code of $y$ is preserved as we maximize the product $xy$ to obtain $\RD(xy)$.
\end{proof}

\begin{theorem}[Generalized Pieri Rule]
\label{thm:genPieri}
Let $x$ be an affine permutation with maximal right decomposition $x=d_{\vec{A}}$ and $k$-code $\alpha=\RD(x)$.  Let $B\subsetneq I$.  Suppose the product $d_Bx\neq 0$.  Then the skew composition $(\RD(d_Bx), \RD(x))$ is a horizontal strip and the skew composition $(\RI(d_Bx), \RI(x))$ is a vertical strip.
\end{theorem}
\begin{proof}
We see that $\RD(x)\subset \RD(d_Bx)$ by Proposition~\ref{prop:containment}.

It is easier to show that the skew composition $(\LD(d_Bx), \LD(x))$ has no more than one box in each column.  The result then follows from Proposition~\ref{prop:colPermutation}, which states that the columns of $\LD(d_Bx)$ are a permutation of the columns of $\RD(d_Bx)$.  Furthermore, showing that $(\LD(d_Bx), \LD(x))$ has no more than one box in each column is equivalent to showing that $(\RD(xd_B), \RD(x))$ has no more than one box in each column.  Thus, we will focus on proving this statement.

To prove this statement, we stack $\RD(x)$ on $\RD(d_B)$ to form a skew $k$-code $(\beta, \alpha)$, and maximize the product using two-row moves.  We note that by Lemma~\ref{lem:2rowEmptyCol} there must be an empty column in $(\beta, \alpha)$ or else the product would not be reduced.  One may then use an algorithm similar to the algorithm in the proof of Lemma~\ref{lem:2rowReduction} to maximize the product and obtain $\RD(xd_B)$.  When $i\in B$ but $i-1\not \in A_1$, (so we have a $\tp$ state), we can use a sequence of two-row moves to move an entire column of $d_A$ downward.  The two types of move needed - iterated commutation moves and iterated chute moves, are illustrated in Figure~\ref{fig.genPieri}.  Otherwise, the algorithm is exactly as in Lemma~\ref{lem:2rowReduction}.

\begin{figure}
  \begin{center}
  \includegraphics[scale=.9]{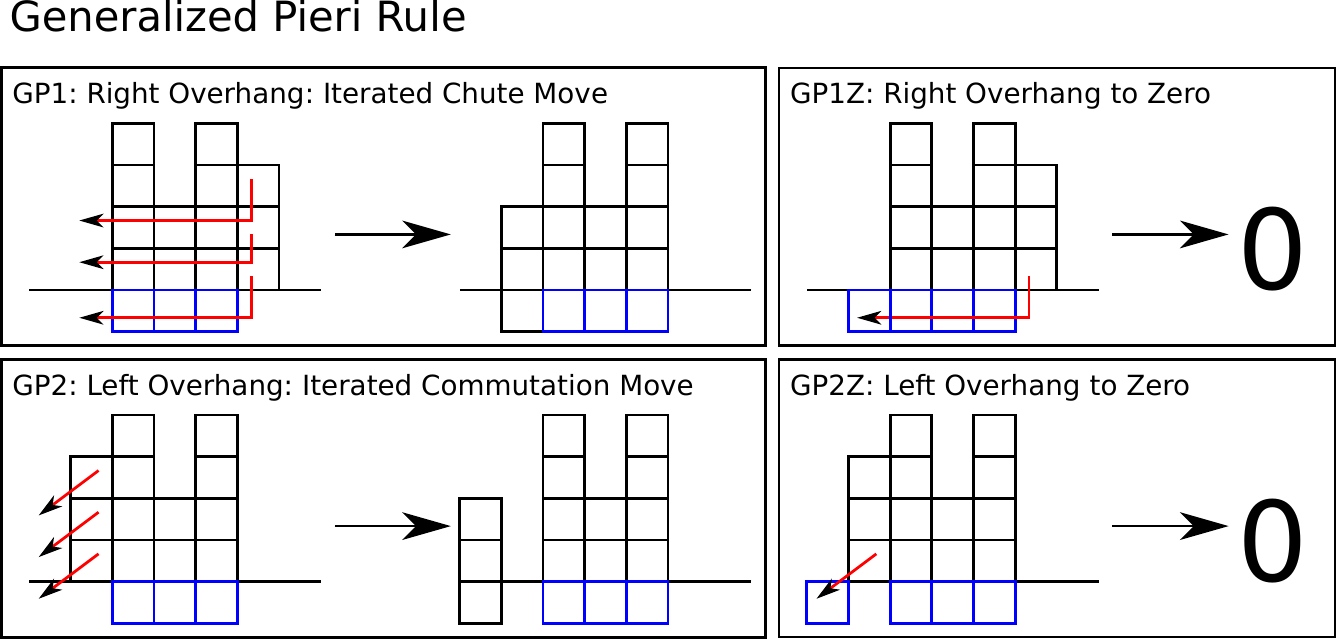}
  \end{center}
  \caption{Multi-row moves used in the proof of the Generalized Pieri Rule, for finding the $k$-code for the product $xd_B$.  The blue boxes represent residues in $d_B$, and the black boxes represent the $k$-code $\RD(x)$.  The empty column appears on the left; reading to the right, we may move down columns appearing over an empty space by applying iterated commutation or chute moves.}
  \label{fig.genPieri}
\end{figure}

The proof that $(\RI(d_Bx), \RI(x))$ is a vertical strip is similar.
\end{proof}

\section{Multiplication of Cyclically Increasing Elements by Cyclically Decreasing Elements.}
\label{sec:incrdecr}

In this section we investigate products of cyclically increasing with cyclically decreasing elements.  We focus in particular on product $u_Bd_A$ with connected $A, B\subsetneq I$, since non-connected cyclic elements are commutative products of connected elements.

\begin{lemma}
Let $A, B \subsetneq I$ with $|B|+|A|\geq k+1$. Then $u_B d_A$ is $i$-dominant if and only if $A$ is $i$-dominant and $B$ is $(i-1)$-dominant.
\end{lemma}
\begin{proof}
Assume $u_Bd_A$ is $i$-dominant.  Then it is clear that $A$ must be $i$-dominant (and thus connected).

Let $A=\{i, i+1, \ldots, l \}$. In order for the product $u_Bd_A$ to be $i$-dominant, we must have $D_R(B)\subseteq \{i-1, l+1\}$, and thus $B$ has at most two connected components.  However, since $u_B$ is cyclically increasing, if there is a component with right descent $l+1$, that component must have cardinality one, or else there will be a braid relation in the product $u_Bd_A$ creating a right descent at $l+1\neq i$.  

Now if $|B|+|A|> k+1$, we have $B \cap A \neq \emptyset$.  We have $|A|<k+1$, so $|B|\geq 2$.  In this case, $B$ must have a single connected component because the component with right descent $i-1$ must be large enough to overlap with $A$, since the component with right descent $l+1$ has cardinality 1.  But then $l+1$ is in the component with right descent $i-1$, implying that there was only one component to begin with.

If $|B|+|A|=k+1$, by similar reasoning, we have $B=I\setminus A$, and is thus connected with right descent $i-1$.  

For the reverse direction, we associate with $B$ the $k$-bounded partition $(1^{|B|})$, and we associate with $A$ the $k$-bounded partition $(|A|)$.  Then we consider the product $e_{|B|}h_{|A|} = s_{(1^{|B|})}^{(k)}s_{(|A|)}^{(k)}$.  By the forward direction, any $0$-dominant element in this product is of the form $u_{B'} d_{A'}$ for some $A', B'\subsetneq I$, where $A'$ is $0$-dominant, $B'$ is $k$-dominant, $|A'|=|A|$, and $|B'|=|B|$.  (If these last two conditions did not hold, we could find such an expression for the element because $u_{B'} d_{A'}$ is a summand in the product $e_{|B|}h_{|A|}$.)  But this specifies $A'$ and $B'$ completely.  Thus, there is only one summmand in $e_{|B|}h_{|A|}$ when expressed as a sum of $k$-Schur functions.  As such, there is only one $i$-dominant term, and it may be obtained by applying the Dynkin diagram automorphism $i$ times.  This exactly recovers the sets $A$ and $B$, and implies that the product $u_Bd_A$ is reduced and $i$-dominant.
\end{proof}

In proving this Lemma, we have also proved the following Corollary:
\begin{corollary}
Suppose $\lambda$ splits into two partitions $(1^p)$ and $(q)$, with $p+q\geq k+1$.  Then 
\[
s_\lambda^{(k)} = s_{(1^p)}^{(k)}s_{(q)}^{(k)}.
\]
\end{corollary}

\begin{corollary}
\label{cor:fatMove}
If $u_Bd_A$ is $i$-dominant with $|B|+|A|\geq k+1$, we have:
\[
u_Bd_A = d_{A+1} u_{B+1},
\]
where $S+1$ is the set obtained by adding 1 to each element of $S \subset I$.  In particular, $u_B$ is $i-1$-dominant.
\end{corollary}
\begin{example}
This calculation is easiest to see with a particular example; the general case is identical.  Suppose $I=\{0,1,2,3,4,5,6\}$ and $A=\{0,1,2,3,4\}, B=\{3,4,5,6\}$.  Then:
\begin{eqnarray*}
u_Bd_A &=& a_{3456}a_{43210} \\
       &=& a_{345643210} \\
       &=& a_{345432160} \\
       &=& a_{543452160} \\
       &=& a_{543214560} \\
       &=& a_{54321}a_{4560} \\
       &=& d_{A+1}u_{B+1}.
\end{eqnarray*}
\end{example}
\begin{proof}[Proof of Corollary.]
This follows directly from the Lemma and a simple computation. By the Lemma, $A=\{i, i+1, \ldots, l\}$ and $B=\{j, j+1, \ldots, i-1,\}$ are connected, and so $D_R(u_B)=\{i-1\}$.  Then:
\[
u_Bd_A= (a_j a_{j+1} \cdots a_{i-2} a_{i-1}) (a_l a_{l-1} \cdots a_{i+1} a_i).
\]
If $|B|+|A|=k+1$, we have $j=l+1$.  Then using the commutation relations:
\begin{eqnarray*}
u_Bd_A &=& (a_{l+1} a_{l+2} \cdots a_{i-2} a_{i-1}) (a_l a_{l-1} \cdots a_{i+1} a_i) \\
       &=& ( a_{l+1} a_l a_{l-1} \cdots a_{i+1} ) ( a_{l+2} \cdots a_{i-2} a_{i-1} a_i ) \\
       &=& d_{A+1}u_{B+1}.
\end{eqnarray*}
If $|B|+|A|>k+1$, then $B$ and $A$ must overlap.  Thus we have $l\in B$.  In the following computation, we use a sequence of subscripts to indicate the product of $a_i$'s.  (So, for example, $a_{1,2,3}=a_1a_2a_3$.)  (The computation uses the extended braid relation, Lemma~\ref{lem:extendedBraidRelation}.)  Then:
\begin{eqnarray*}
u_Bd_A &=& a_{j,j+1,\cdots,l,l+1,l+2,\cdots,i-2,i-1} a_{l,l-1,\cdots,i+1,i} \\
       &=& a_{j,j+1,\cdots,l,l+1}a_{l+2,\cdots,i-2,i-1} a_{l,l-1,\cdots,i+1}a_i \\
       &=& a_{j,j+1,\cdots,l,l+1} a_{l,l-1,\cdots,i+1} a_{l+2,\cdots,i-2,i-1} a_i \\
       &=& a_{j,j+1,\cdots,l,l+1} a_{l,l-1,\cdots,j+1,j,j-1,\cdots,i+1} a_{l+2,\cdots,i-2,i-1} a_i \\
       &=& a_{j,j+1,\cdots,l,l+1,l,\cdots,j+1,j}a_{j-1,\cdots,i+1} a_{l+2,\cdots,i-2,i-1} a_i \\
       &=& a_{l+1,l,\cdots,j+1,j,j+1,\cdots,l,l+1}a_{j-1,\cdots,i+1} a_{l+2,\cdots,i-2,i-1} a_i \\
       &=& a_{l+1,l,\cdots,j+1,j}a_{j+1,\cdots,l,l+1}a_{j-1,\cdots,i+1} a_{l+2,\cdots,i-2,i-1} a_i \\
       &=& a_{l+1,l,\cdots,j+1,j}a_{j-1,\cdots,i+1} a_{j+1,\cdots,l,l+1} a_{l+2,\cdots,i-2,i-1,i} \\
       &=& d_{A+1}u_{B+1}.
\end{eqnarray*}
This completes the proof.
\end{proof}

\subsection{Products of Cyclically Increasing and Decreasing Elements}
\label{ssec:interactions}

We catalog the result of multiplying $u_Bd_A$ for any connected $A, B \subsetneq I$.  First we fix some notation.

\begin{definition}
Let $B\subsetneq I$ be connected, with $B=\{i, i+1, \ldots, j-1, j\}$.  Set:
\begin{eqnarray*}
B^+ &=& B \cup \{ j+1 \} \\
B^- &=& B \setminus \{ j \} \\
B_+ &=& B \cup \{ i-1 \} \\
B_- &=& B \setminus \{ i \}
\end{eqnarray*}
Additionally, let the sets with both subscripts and superscripts be defined in the obvious way.  (So that $B^+_- = (B\cup \{j+1\}) \setminus \{i\}$, for example.)
\end{definition}

\begin{lemma}
\label{lem:breakCycle}
Let $B\subsetneq I$ be connected, with $B=\{i, i+1, \ldots, l, l+1, \ldots, j-1, j\}$.  
Set $B_1=\{i, i+1, \ldots, l\}$ and $B_2=\{l, l+1, \ldots, j\}$.  Then:
\begin{eqnarray*}
u_B &=& u_{B_1} u_{B_2} \\
d_B &=& d_{B_2} d_{B_1}. \\
\end{eqnarray*}
\end{lemma}
\begin{proof}
This is follows immediately from the definitions of $u_B$ and $d_B$.
\end{proof}

\begin{proposition}
\label{prop:simpleProducts}
Let $B=\{i, i+1, \ldots, j-1, j\}$ and $A=\{p, p+1, \ldots, q-1, q\}$.  Then we have the following:
  \begin{equation}
    u_B d_A =
    \begin{cases}
      0              & \text{$j= q$}\\
      d_A u_{B^+_-}  & \text{if $B\subset A$ and $j\neq q$}\\
      d_{A^+_-} u_B  & \text{if $A\subset B$ and $j\neq q$}\\
      d_{A_-}u_{B^+} & \text{if $A\cap B=\emptyset, i=p-1, j\neq q+1$}\\
      d_{A^+}u_{B_-} & \text{if $A\cap B=\emptyset, i\neq p-1, j=q+1$}\\
      d_{A^+_-}u_{B^+_-} & \text{if $A\cap B=\emptyset, i=p-1, j=q+1$}\\
      d_Au_B         & \text{if $A\cap B=\emptyset, i\neq p-1, j\neq q+1$}\\
    \end{cases}
  \end{equation}
\end{proposition}
\begin{proof}
These all follow from straight-forward computations and the extended braid relation, Lemma~\ref{lem:extendedBraidRelation}.

These computations are nearly identical to the computation in the proof of Corollary~\ref{cor:fatMove}, and are thus omitted here.
\end{proof}

\begin{proposition}
\label{prop:updowndownup}
Let $B, A\subsetneq I$, with both $B$ and $A$ connected.  Then there exist connected sets $B', A' \subsetneq I$ with $|A'+B'|=|A+B|$ such that
\[
u_Bd_A = d_{A'} u_{B'}.
\]
Furthermore, the pair $(A', B')$ is one of $(A, B), (A, B^+_-), (A^+_-, B), (A^+, B_-)$, or $(A_-, B^+)$.
\end{proposition}
\begin{proof}
One may use Lemma~\ref{lem:breakCycle} and Proposition~\ref{prop:simpleProducts} to derive arbitrary products $u_Bd_A$ by taking $B=(B\cap A)\cup (B\setminus (B\cap A))$.  Then the proof comes down to checking six additional cases, which all work out.  These additional cases are the `overlapping' cases where $A\cap B\neq \emptyset$, but $B$ not contained in $A$ and vice versa.
\end{proof}

In particular, consider the product $w$ for $A=\{0, 1, \ldots, 1\}$ and $B_i$ connected with $|B_i|>|B_{i+1}|$ for each $i$ given by:
\[
w := u_{B_l}u_{B_{l-1}}\ldots u_{B_1} d_A = d_{A'} u_{B_l'}u_{B_{l-1}'}\ldots u_{B_1'}
\]
Then $w$ is $0$-dominant only if $u_{B_1'}u_{B_2'}\ldots u_{B_l'}$ is $0$-dominant.

\section{The $k$-Littlewood-Richardson Rule for Split\newline $k$-Schur Functions}
\label{sec:lwrule}

Our goal in this section is to prove a special case of the Littlewood-Richardson rule for $k$-Schur functions, as described in the introduction.  The proof will rely heavily on the maximal decomposition of affine permutations as well as multiplication of cyclically increasing and decreasing elements.

First, we reformulate the splitting condition for cores in terms of the sizes of rows and columns of the associated bounded partitions.

\begin{lemma}
Let $\lambda$ be a $k$-bounded partition whose associated $k+1$-core $\core(\lambda)$ splits into $k+1$-cores $\core(\mu)$ and $\core(\nu)$.  Then for any $i,j$, we have $\mu^{(k)}_i + \nu_j \geq k+1$.  
\end{lemma}
\begin{proof}
Suppose $\mu^{(k)}$ has $m$ parts and $\nu$ has $n$ parts.  We show that $\mu^{(k)}_m + \nu_n \geq k+1$; the statement then holds for arbitrary $i,j$ since $\mu^{(k)}$ and $\nu$ are partitions, so that: 
\[
\mu^{(k)}_i + \nu_j \geq \mu^{(k)}_m + \nu_n \geq k+1.
\]

Diagonally stacking the cores $\core(\mu)$ and $\core(\nu)$ yields the core $\core(\lambda)$.  By pushing the $k$-boundary of $\core(\mu)$ down, we obtain the $k$-column bounded partition whose transpose is the $k$-bounded partition $\mu^{(k)}$.  Pushing the $k$-boundary of $\core(\nu)$ to the left, we obtain $\nu$.  (See Figure~\ref{fig.splitCore2} for an example.)  All of the boxes in the last column of $\core(\mu)$ have hook $\leq k$, and are thus in the boundary $\partial_k(\core(\mu))$.  Likewise for the boxes in the top row of $\core(\nu)$.  But $\core(\lambda)$ splits at a box with hook $\geq k+1$, so we have $\mu^{(k)}_m + \nu_n \geq k+1$.
\end{proof}

\begin{figure}
  \begin{center}
  \includegraphics[scale=1]{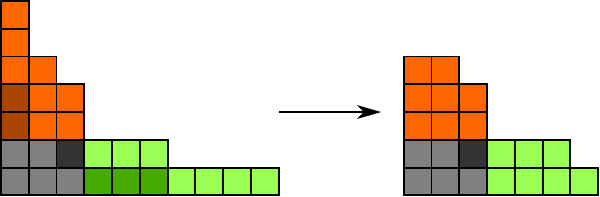}
  \end{center}
  \caption{A split $5$-core.  The core splits into the orange core $\core(\mu)$ and the green core $\core(\nu)$.  The $5$-boundary is given by the lighter-colored boxes.  Down-justifying the orange core gives a $4$-column bounded partition, and left-justifying the green partition gives a $4$-bounded partition.  The sum of any light orange column and light green row is $\geq 5$, because the dark grey box has hook $>5$.}
  \label{fig.splitCore2}
\end{figure}

Suppose $\lambda$ splits into factors $\mu$ and $\nu$.  We will express summands in $s_\mu^{(k)}$ as products of cyclically increasing elements and summands in $s_\nu^{(k)}$ as a product of cyclically decreasing elements.  
To find $k$-Littlewood-Richardson coefficients, we need to identify $0$-dominant terms in the product $s_\mu^{(k)}s_\nu^{(k)}$.  
In any product $wv$ for $w, v\in \nilcoxmon$, we have $D_R(v) \subset D_R(wv)$.  Thus, if $v$ is not $0$-dominant then the product $wv$ cannot be $0$-dominant.  Since $s_\nu^{(k)}$ has a unique $0$-dominant summand $u_\nu$, we consider products $u_{\vec{A}} u_\nu$, where $u_{\vec{A}}$ appears in $s_\mu^{(k)}$.  We then need to answer two questions:
\begin{itemize}
\item For which $\vec{A}$ is the product $u_{\vec{A}} u_\nu$ $0$-dominant?
\item Which of these $u_{\vec{A}}$ appear as summands in $s_\mu^{(k)}$?
\end{itemize}

\begin{definition}
Let $x,y \in \nilcoxmon_k$.  Then $x$ is \emph{left-compatible} with $y$, which we denote $x\vdash y$, if $xy\neq 0$ and $D_R(xy)=D_R(y)$.
\end{definition}

Then, when $\nu\neq \emptyset$, the $k$-Littlewood-Richardson coefficients may be expressed as:
\[
c^{\lambda}_{\mu, \nu} = \sum [x]s_\mu^{(k)},  
\]
where the sum is over $x \vdash u_{\nu}$ such that $x u_{\nu}=u_{\lambda}$.

\begin{lemma}
\label{lem:mainMachine}
Suppose $u_{\vec{B}} = u_{B_m} \cdots u_{B_1}$ is a maximal cyclically increasing product of shape $\lambda = (|B_1|, \ldots, |B_m|)$ and $d_{\vec{A}}=d_{A_n}\cdots d_{A_1}$ satisfies $|A_j|+|B_i|\geq k+1$ for all $i, j$.  Then $u_{\vec{B}} d_{\vec{A}}$ is $i$-dominant if and only if $d_{\vec{A}}$ is $i$-dominant and $u_{\vec{B}}$ is $(i-n)$-dominant.  

In this case, we also have:
\[
u_{B_m} \cdots u_{B_1} d_{\vec{A}} = d_{\vec{A}+m} u_{B_m+1} \cdots u_{B_1+1}.
\]
\end{lemma}
\begin{proof}
For the reverse direction, let $\lambda=\sh(\vec{B})$ and $\mu=\sh(\vec{A})$ be the bounded partitions associated to the elements $u_{\vec{B}}$ and $d_{\vec{A}}$.  Then diagonally stacking the $k+1$-cores $\core(\lambda^{(k)})$ and $\core(\mu)$ yields a $k+1$-core $\nu$ that splits into $\lambda^{(k)}$ and $\mu$.  The $0$-dominant element associated to $\nu$ is equal to $\Psi^{(-i)}(u_{\vec{B}} d_{\vec{A}})$, and so we see that the product is $i$-dominant.

The forward direction is more complicated.  We see immediately that $d_{\vec{A}}$ must be $i$-dominant in order for the product to be $i$-dominant.  Thus, we induct on $m$, the number of parts of $\vec{B}$.

For the base case, $\vec{B}=B_1$, and our assumption is $u_{B_1}\vdash d_{\vec{A}}$.  Since $d_{\vec{A}}$ is $i$-dominant, we have $A_n$ connected, so let $A_n=[i-n+1, j]$.  (The features of the base case are illustrated in Figure~\ref{fig.example-mainBaseCase}.)

Let $L=[p,q]$ be any connected component of $B_1$ with $D_R(L) \neq \{i-n\}$.  (We will argue by contradiction to show that no such $L$ can exist.)  We observe that $D_R(u_{B_1})\cap\{i-n+1, \ldots, j\}=\emptyset$, or else the product is zero or non-dominant.  Thus, $q\not \in \{i-n+1, \ldots, j\}$, so that $q\in [j+1, i-n-1]$.  

In fact, we claim that $L\cap \{i-n, \ldots, j\}=\emptyset$.  Otherwise, $L$ must contain $j$, as $u_L$ is an increasing product.  If $j\in L$, set $d_{\vec{A'}}$ be the maximal decomposition of $a_{j+1}\cdots a_q d_{\vec{A}}$.  Since $[j+1,q]\cap \{i-n, \ldots, j\}=\emptyset$, $\vec{A'}$ has the same number of parts as $\vec{A}$.  Then we consider $a_jd_{\vec{A'}}$.  If $j+1\in A'_{m}$, multiplication by $a_j$ starts a new row in the $k$-code $\RD(d_{\vec{A'}})$, which creates a new right descent.  If $j+1\not \in A'_m$, then $a_jd_{\vec{A'}}=0$, since $j\in D_R(d_{A_m})$.  Thus, $j\not \in L$, and so $L\cap \{i-n, \ldots, j\}=\emptyset$.

If there is no connected component of $B_1$ with right descent $i-n$, we then have $B_1\subset I \setminus \{i-n, \ldots, j\}$.  But there are $|A_n|+1$ elements in $\{i-n, \ldots, j\}$, so that $|B_1|\leq k+1 - (|A_n|+1)$, so that $|B_1|+|A_n|\leq k$, contradicting the assumption that $|B_1| + |A_n| \geq k+1$.

On the other hand, suppose $B_1$ has a connected component $C$ with right descent $i-n$.  Then $C=\{r, \ldots i-n-1, i-n\}$.  If there are other connected components, we know that they are contained in $[j+1, i-n-1]$.  Since $C$ contains $i-n$, no other connected component may contain any elements in $[r-1, i-n+1]$, or else that component would be connected to $C$.  As a result, if $C\cap [i-n+1, j]\neq \emptyset$ there can be no other connected components.  Thus, all other components are subsets of $[j+1, r-2]$.  
But if there are other connected components, we then have $|C \cup [j+1, r-2] \cup A_n| = k$, so that $|B|+|A_n|<k+1$, contrary to assumption.

Thus, there are no connected components of $B_1$ with right descent other than $(i-n)$.  As a result, $B_1$ is $(i-n)$-dominant, as desired.
(For an example, see Figure~\ref{fig.example-mainBaseCase}.)

\begin{figure}
  \begin{center}
  \includegraphics[scale=1.25]{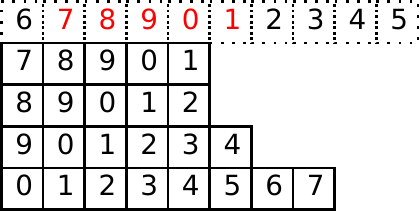}
  \end{center}
  \caption{An example with $k=9$ for the base case of Lemma~\ref{lem:mainMachine}.  On the left is the $k$-code is for a $0$-dominant element $d_{\vec{A}}$ of length $24$; we consider a $0$-dominant product $u_Bd_{\vec{A}}$ where $|B|\geq 5$.  Left-multiplying by any of the generators in $A_n=\{a_7, a_8, a_9, a_0, a_1\}$ will either create a descent or kill the element.  Thus, elements in $A_n$ cannot be right descents of $u_B$, and by the argument in the proof, cannot be in any connected component of $B$ whose right descent is not $6$.  \newline
  On the right is the Dynkin diagram with $k=9$.  The set $A_n$ is highlighted in red, and an increasing subset $C$ with right descent $6$ is highlighted in cyan.  The proof for the base case argues that any other connected component of $B$ would then have to be contained in the set $\{2, 3\}$, but then $|B|+|A_n|\leq 9$, contradicting the assumption that $|B|+|A_n|\geq 10=k+1$.  We may then observe that $B$ must be connected with right descent 6.}
  \label{fig.example-mainBaseCase}
\end{figure}


For the inductive step, we assume that $x=u_{B_{m-1}}\cdots u_{B_1}$ is $(i-m)$-dominant and $u_{\vec{B}}$ maximal.  Set $\vec{B}^- =\{B_1, \ldots, B_{m-1}\}$. We consider $D_R(u_{B_m})$.  Maximality of $\vec{B}$ implies that each $B_n$ is a subset of $B_{n-1}+1$.

We observe that if $u_{B_m}$ is not connected with right descent $i-n+m$, there exists an alternate factorization $x=u_{\vec{B'}}$ with 
$B_1'=\{j-(m-1) \mid j \in B_m\}$.  This occurs because each connected component of $B_m$ contributes a right descent to $u_{\vec{B}}$.  Factoring out this right descent on the right leaves the next element in the connected component, and so on.  This $B_1'$ must have $u_{B_1'}d_{\vec{A}}$ $0$-dominant, or else $u_{\vec{B}}d_{\vec{A}}$ will not be zero-dominant.  Then by the base case, this $B_1'$ must be connected with right descent $i-n$.  But this means that $B_m$ was connected with right descent $i-n+m$.

\begin{figure}
  \begin{center}
  \includegraphics[scale=1]{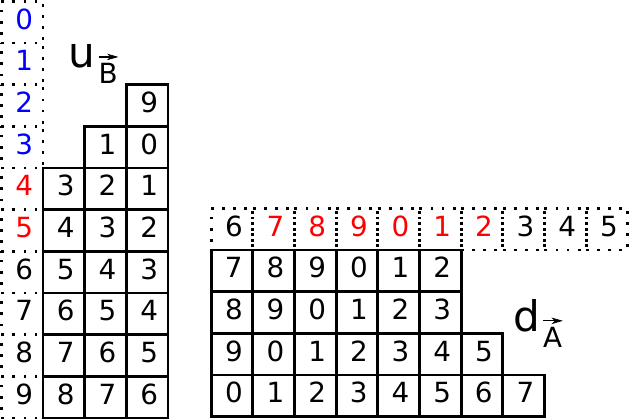}
  \end{center}
  \caption{An example with $k=9$ for the induction step of Lemma~\ref{lem:mainMachine}.  The $k$-code on the left represents the element $u_{\vec{B}}$ and the $k$-code on the right is for an element $d_{\vec{A}}$; consider a $0$-dominant product $u_Bu_{\vec{B}}d_{\vec{A}}$ where $|B|\geq 5$.  Left-multiplying $d_A$ by any of the generators in $S=\{a_7, a_8, a_9, a_0, a_1\}$ will either create a descent or kill the element.  $B$ cannot contain any of the blue residues or else maximality of $u_Bu_{\vec{B}}$ will be violated.  If $a_5$ is a right descent of $u_B$ then $a_2$ will be a right descent of $u_Bu_{\vec{B}}$, and is thus disallowed (and colored red in the diagram).  If $B_m$ is disconnected, (for example, set $B=\{4,5,6,8,9\}$), we can re-express the element as $u_{\vec{B}}=xu_{1,2,3,5,6}$.  But this violates the base case.  Then since $|B|\geq 5$, $B$ must be connected with $D_R(u_B)=\{9\}$.}
  \label{fig.example-mainInduct}
\end{figure}
\end{proof}

\begin{lemma}
\label{lem:mainResult}
Suppose a $k$-bounded partition $\lambda$ splits into two components, $\mu$ and $\nu$.  Then $s_{\mu}^{(k)}s_\nu^{(k)}=s_\lambda^{(k)}$.
\end{lemma}
\begin{proof}

We induct on the number of parts in $\mu$, considering elements of $\nilcoxmon_k$ appearing with non-zero coefficient in $s_{\mu}^{(k)}$ as products of cyclically increasing elements, and those appearing in $s_{\nu}^{(k)}$ as decreasing elements.  Then any $x$ with $[x]s_{\mu}^{(k)}\neq 0$ has a cyclically increasing expansion $x=u_{B_m}\cdots u_{B_1}$ where $|B_i|=\mu_i^{(k)}$.

Let $a_\nu$ denote the unique $0$-dominant term in $s_{\nu}^{(k)}$, with maximal cyclically decreasing product $a_\nu=d_{A_n}\cdots d_{A_1}$.  Then we claim that there exists a unique summand $x$ in $s_{\mu}^{(k)}$ such that $x\vdash a_\nu$, and furthermore that $x$ is $(k+1-n)$-dominant.  By the splitting condition we have $|B_i|+|A_j|\geq k+1$ for all $i, j$.  To do this, we will induct on the number of parts of $\mu^{(k)}$.

For the base case, $\mu^{(k)}$ has a single part, so that $s_{(l)}^{(k)}=e_{l} = \sum_{|A|=l} u_A$.  In this case, every summand is maximal.  Then by Lemma~\ref{lem:mainMachine}, if $x\vdash a_\nu$, then $x$ is $(k+1-n)$-dominant.  There is a unique such element in $s_{(l)}^{(k)}$, so the product $s_{(l)}^{(k)}s_\lambda^{(k)}$ has a single $0$-dominant summand, as desired.

For the induction step, we suppose the statement holds for any $k$-bounded partition $\rho$ with $\leq m$ parts.  Let $\mu^{(k)} = \rho \cup (l)$ be a $k$-bounded partition with $m+1$ parts.  We consider the product:
\[
e_l s_\rho^{(k)} = s_\mu^{(k)} + \sum_\kappa s_\kappa^{(k)},
\]
according to the Pieri rule.  We recall that each $\kappa \succ \mu$.  Finally, let $a_\rho$ be the unique $(k+1-n)$-dominant summand in $s_\rho^{(k)}$.

Claim 1: For any $x$ a summand in $e_l s_\rho^{(k)}$, if $x\vdash a_\nu$ then $x=u_Ca_\rho$ with $|C|=l$.

We are interested in $0$-dominant elements in the product $e_l s_\rho^{(k)} s_\nu^{(k)}$.  Recall that for any elements $p,q$ in a Coxeter group, we have $D_R(q)\subset D_R(pq)$.  Then any $0$-dominant term in $s_\rho^{(k)} s_\nu^{(k)}$ is of the form $p a_\nu$ for some $p$.  But by the inductive hypothesis, we see that the only non-zero summand $p$ in $s_\rho^{(k)}$ with $p\vdash a_\nu$ is $a_\rho$.  The claim then follows immediately.


Claim 2: For any $x\vdash a_\nu$, we have 
\[
[x] e_l s_\rho^{(k)} = 1.
\]

By first claim, we have $x=u_Ca_\rho$.  If the coefficient were greater than $1$, we would have a second decomposition $u_Da_\rho$.  But then $u_C=u_D$.

Claim 3: Let $x=u_Ca_\rho \vdash a_\nu$, and $a_\nu:=u_{\vec{B}}$.  Then if $u_Cu_{\vec{B}}$ is not maximal, then $[x] s_\mu^{(k)} = 0$.

By the results of Section~\ref{sec:kcastles}, $x$ has a unique maximal decomposition $x=u_{C'}u_{\vec{B'}}$, where $\vec{B'}$ has the same number of parts as $\vec{B}$ and $|C'|\leq|C|$.  By the inductive hypothesis, $u_{\vec{B'}}$ is then $(k+1-n)$-dominant; let it be of shape $\gamma$.  
By Proposition~\ref{prop:containment}, we have $\rho \subset \gamma$.  Set $\gamma^+=\gamma \cup (|C'|)$, which is $\sh(x)$.

By Theorem~\ref{thm:genPieri}, $\gamma^+\setminus \rho$ is a weak strip, so $\gamma^+$ appears in the Pieri rule expansion of $e_ls_\rho^{(k)}$.  Furthermore, $[x]s_{\gamma^+}^{(k)}=1$, by Theorem~\ref{thm:maxhCoefficient}.  
Then we observe that: 
\begin{eqnarray*}
1 &=& [x] e_l s_\rho^{(k)} \\
  &=& [x] (s_\mu^{(k)} + s_{\kappa^+}^{(k)} + \sum_\kappa s_\kappa^{(k)})\\ 
  &=& [x]s_\mu^{(k)} + [x]s_{\kappa^+}^{(k)}+ \sum_\kappa [x]s_\kappa^{(k)}.
\end{eqnarray*}
All of these coefficients are $\geq 0$, and $[x]s_{\kappa^+}^{(k)}=1$, so $[x]s_\mu^{(k)}=0$.  

Thus, when $x\vdash a_{\nu}$ and $[x]s_\mu^{(k)}>0$, we have $x=u_Cu_{\vec{B}}$ is maximal.  Then by Lemma~\ref{lem:mainMachine}, $x$ is $(k+1-n)$-dominant.  There is a unique such element in $s_\mu^{(k)}$, which completes the proof.
\end{proof}

\begin{theorem}
\label{thm:mainResult}
Suppose $\lambda$ splits into components $\mu_i$.  Then 
\[
s_\lambda^{(k)} = \prod s_{\mu_i}^{(k)}.
\]
\end{theorem}
\begin{proof}
This follows from successive application of Lemma~\ref{lem:mainResult}.
\end{proof}

\bibliographystyle{amsalpha}

\newcommand{\etalchar}[1]{$^{#1}$}
\providecommand{\bysame}{\leavevmode\hbox to3em{\hrulefill}\thinspace}
\providecommand{\MR}{\relax\ifhmode\unskip\space\fi MR }
\providecommand{\MRhref}[2]{%
  \href{http://www.ams.org/mathscinet-getitem?mr=#1}{#2}
}
\providecommand{\href}[2]{#2}

\end{document}